\documentclass[a4paper,11pt,final]{amsart}
\usepackage{ifdraft}

\ifdraft{\usepackage[draft]{showkeys}}{\usepackage[final]{showkeys}}

\usepackage[T1]{fontenc} \usepackage[utf8x]{inputenc}

\usepackage{amsmath,amsfonts,eucal,amsthm,amssymb,extarrows,upgreek,tensor,mathrsfs,textcomp,comment,mathtools,bm,bbm,braket}
\usepackage{mathscinet}
\usepackage[shortlabels]{enumitem}

\numberwithin{equation}{section}

\usepackage[usenames,dvipsnames]{xcolor} \usepackage{tikz}
\usetikzlibrary{patterns,matrix,through,arrows,decorations.pathreplacing}

\usepackage{graphicx} \usepackage{wrapfig}

\usepackage{todonotes} 

\DeclareSymbolFont{usualmathcal}{OMS}{cmsy}{m}{n}
\DeclareSymbolFontAlphabet{\mathucal}{usualmathcal}

\newtheoremstyle{myplain} {10pt}
{10pt}
{\itshape}
{}
{\bfseries}
{.}
{.5em}
{}

\theoremstyle{myplain}
\newtheorem{theorem}{Theorem}[section]
\newtheorem*{theorem*}{Theorem}
\newtheorem{lemma}[theorem]{Lemma}
\newtheorem{prop}[theorem]{Proposition}
\newtheorem{corollary}[theorem]{Corollary}

\newtheoremstyle{mydefinition} {10pt}
{10pt}
{\itshape}
{}
{\bfseries}
{.}
{.5em}
{}

\theoremstyle{mydefinition}
\newtheorem{definition}[theorem]{Definition}

\newtheoremstyle{myexample} {6pt}
{6pt}
{}
{}
{\scshape}
{.}
{.5em}
{}

\theoremstyle{myexample}

\newtheoremstyle{myremark} {6pt}
{6pt}
{}
{}
{\scshape}
{.}
{.5em}
{}

\theoremstyle{myremark}
\newtheorem{remark}[theorem]{Remark}

\newcommand{\N}{\mathbb{N}}
\newcommand{\Z}{\mathbb{Z}}

\newcommand{\C}{\mathbb{C}}

\newcommand{\gl}{\mathfrak{gl}}
\newcommand{\catO}{\mathcal{O}}

\newcommand{\suchthat}{\,|\,} 
\newcommand{\mapto}{\rightarrow}

\newcommand{\into}{\hookrightarrow}

\DeclareMathOperator{\Hom}{Hom} 
\DeclareMathOperator{\End}{End}

\DeclareMathOperator{\id}{id}

\newcommand{\abs}[1]{\left|#1\right|}

\renewcommand{\epsilon}{\varepsilon}

\renewcommand{\phi}{\varphi}

\newcommand{\calC}{{\mathcal{C}}}

\newcommand{\calY}{{\mathcal{Y}}}

\newcommand{\sfG}{{\mathsf{G}}}

\newcommand{\sfs}{{\mathsf{s}}}

\newcommand{\frakg}{{\mathfrak{g}}}
\newcommand{\frakh}{{\mathfrak{h}}}

\newcommand{\frakn}{{\mathfrak{n}}}

\newcommand{\frakp}{{\mathfrak{p}}}

\newcommand{\bbS}{{\mathbb{S}}}

\newcommand{\rmB}{{\mathrm{B}}}

\newcommand{\ucalH}{{\mathucal H}}

\newcommand{\Mor}{\mathcal{M}\hspace{-0.1ex}\mathrm{or}}
\newcommand{\Ob}{\mathcal{O}\hspace{-0.1ex}\mathrm{b}}

\newcommand{\acts}{\mathop{\;\raisebox{1.7ex}{\rotatebox{-90}{$\circlearrowright$}}}\;}
\newcommand{\actsb}{\mathop{\;\raisebox{0ex}{\rotatebox{90}{$\circlearrowleft$}}}\;}

\newcommand{\bomega}{{\boldsymbol{\omega}}}

\newcommand{\VB}{\mathsf{V\hspace{-5.5pt}VB}}
\newcommand{\Wenzl}{\mathsf{V\hspace{-5.5pt}V}}

\newcommand{\Br}{{\mathrm{Br}}}
\newcommand{\uBr}{\underline\Br}
\newcommand{\uVB}{\underline\VB}
\newcommand{\udelta}{{\underline \delta}}

\newcommand{\hats}{{\hat s}}
\newcommand{\hate}{{\hat e}}

\newcommand{\Seq}{{\mathrm{Seq}}}

\usepackage{hyperref,bookmark}

\hypersetup{ colorlinks = false, urlcolor = false, pdfauthor =
  {Antonio Sartori}, pdfkeywords = {}, pdftitle = {Generalized walled
    Brauer algebra}, pdfsubject = {},
  pdfpagemode = UseNone, bookmarksopen = true, bookmarksopenlevel = 1,
  pdfdisplaydoctitle = true }

\title{The degenerate affine walled Brauer algebra}
\author{Antonio Sartori}
\address{Mathematisches Institut\\Endenicher Allee 60\\Universit\"at Bonn\\53115 Bonn, Germany}
\email{sartori@math.uni-bonn.de}
\urladdr{http://www.math.uni-bonn.de/people/sartori}
\date{\today}
\thanks{This work has been supported by the Graduiertenkolleg 1150, funded by the Deutsche Forschungsgemeinschaft.}

\begin{document}

\begin{abstract}
  We define a degenerate affine version of the walled Brauer algebra, that has the same role plaid by the degenerate affine Hecke algebra for the symmetric group algebra. We use it to prove a higher level mixed Schur-Weyl duality for $\gl_N$. We consider then families of cyclotomic quotients of level two which appear naturally in Lie theory and we prove that they inherit from there a natural grading and a graded cellular structure.
\end{abstract}

\maketitle
\tableofcontents

\section{Introduction}
\label{sec:introduction}

The classical Schur-Weyl duality \cite{MR0000255} is a basic but very important result in representation theory that connects the general Lie algebra $\gl_N$ with the symmetric group $\bbS_r$. If $V \cong \C^N$ is the vector representation of $\gl_N$, Schur-Weyl duality states that the natural action of $\bbS_r$ on $V^{\otimes r}$ is $\gl_N$-equivariant, and hence induces a map
\begin{equation}\label{eq:245}
  \Phi: \C[\bbS_r] \longrightarrow \End_{\gl_N}(V^{\otimes r}),
\end{equation}
which is always surjective, and is injective if and only if $r\geq N$.

Many generalizations and variations of Schur-Weyl duality, in which one replaces for example the Lie algebra $\gl_N$ with some other reductive Lie algebra and/or the representations $V^{\otimes r}$ with other representations, have been studied in the past years. We would like to recall three of them (as a general reference, see \cite[\textsection{}4.2]{MR2522486}).

Firstly, there is a version connecting the semisimple Lie algebra $\frakg$ of type $B$, $D$ (or $C$) and its vector representation $V$ with the Brauer  algebra $\rmB_{r}(N)$ (or $\rmB_r(-N)$) \cite{MR1503378}:
\begin{equation}\label{eq:248}
  \frakg \acts  V^{\otimes r} \actsb \rmB_{r}(\pm N).
\end{equation}

Secondly, there is a mixed version of \eqref{eq:245} in which one considers mixed tensor products of the vector representation $V$ and its dual:
\begin{equation}\label{eq:248}
  \gl_N \acts \big( V^{\otimes r} \otimes V^{*\otimes t} \big) \actsb \Br_{r,t}(N).
\end{equation}
Here $\Br_{r,t}(N)$ is the walled Brauer algebra, a subalgebra of the Brauer algebra $\rmB_{r+t}(N)$ which was introduced independently by Turaev \cite{MR1024455} and Koike \cite{MR991410}.

Finally, there is a higher version of Schur-Weyl duality (cf.\ \cite{MR1652134}) in which one considers the tensor product of the representation $V^{\otimes r}$ with some (possibly) infinite dimensional module $M \in \catO(\gl_N)$:
\begin{equation}\label{eq:250}
  \gl_N \acts \big( M \otimes V^{\otimes r} \big) \actsb \ucalH_r.
\end{equation}
Here $\ucalH_r$ is the degenerate affine Hecke algebra of $\bbS_r$.

The goal of the present paper is to define a degenerate affine version
of the walled Brauer algebra $\Br_{r,t}(N)$ in such a way that we get
a higher version of mixed Schur-Weyl duality that generalizes both
\eqref{eq:248} and \eqref{eq:250}:
\begin{equation}\label{eq:251}
  \gl_N \acts \big( M \otimes V^{\otimes r} \otimes V^{*\otimes t} \big) \actsb \VB_{r,t}(\bomega).
\end{equation}
For technical reasons, we will actually need $M$ to be a highest weight module; the parameter $\bomega$ is a sequence $(\omega_i)_{i \in \N}$ of complex numbers which depend on $M$.

The passage  $V^{\otimes r} \rightsquigarrow V^{\otimes r} \otimes V^{*\otimes t}$ from \eqref{eq:250} to \eqref{eq:251} is quite natural and is motivated for example by the following reason. Brundan and Kleshchev \cite{MR2551762} constructed an explicit isomorphism between cyclotomic quotients of the degenerate affine Hecke algebra $\ucalH_r$ and the KLR algebra $R$; but the KLR algebra $R$ is in some sense only one half of the KLR 2-category $\dot{\mathucal{U}}$ \cite{1206.17015}. The degenerate affine walled Brauer algebra $\VB_{r,t}(\bomega)$ should correspond to the whole $\dot{\mathucal{U}}$ (see also \eqref{eq:253} below).

We remark that in \eqref{eq:248} we may permute the $V$'s and the $V^*$'s. In particular, we have a version of \eqref{eq:248} for each $(r,t)$-sequence $A=(a_1,\ldots,a_{r+t})$, by which we mean a permutation of the sequence
\begin{equation}
  \label{eq:249}
  (\underbrace{1,\ldots,1}_{r},\underbrace{-1,\ldots,-1}_{t}).
\end{equation}
If in \eqref{eq:248} we replace $V^{\otimes r} \otimes V^{*\otimes t}$ by $V^{\otimes A} = V^{a_1} \otimes \cdots \otimes V^{a_{r+t}}$ with the convention $V^1=V$ and $V^{-1}=V^*$, then the walled Brauer algebra $\Br_A(N)$ which acts on the right is of course isomorphic to $\Br_{r,t}(N)$, but has a different natural presentation. If moreover we consider all these sequences $A$ together, then it becomes natural to replace the walled Brauer algebra with the \emph{walled Brauer category} $\uBr_{r,t}(N)$.
In our setting, it is easier to define directly the \emph{degenerate affine walled Brauer category} $\uVB_{r,t}(\bomega)$ and then get the degenerate affine walled Brauer algebras $\VB_A(\bomega)$ as endomorphism algebras inside our category $\uVB_{r,t}(\bomega)$.

Our first main theorem is:
\begin{theorem*}[See Theorem \ref{thm:1}]
  Let $M$ be a highest weight module for $\gl_N$ and $A$ be a $(r,t)$-sequence. Then there is a natural action of the algebra $\VB_A(\bomega)$ on $M \otimes V^{\otimes A}$, which commutes with the action of $\gl_N$.  Here the parameter $\bomega=\bomega(M)$ is determined by the highest weight of the module $M$.
\end{theorem*}

The action given by the theorem is far from being faithful. To get a
faithful action we need to consider cyclotomic walled Brauer algebras
as endomorphism algebras inside a cyclotomic quotient of the category
$\uVB_{r,t}(\bomega)$. In particular, we will study cyclotomic
quotients of level two. Our second main result is:
\begin{theorem*}[See Theorem \ref{thm:2}]
  Let $m,n,N,r,t \in \N$ with $m+n=N$ and $m,n \geq r+t$. Let $\frakp \subset \gl_N$ be the standard parabolic subalgebra of
  $\gl_N$ corresponding to the two-blocks Levi $\gl_m \oplus \gl_n \subset \gl_N$. For $\delta \in \Z$ with $\delta \neq m,n$ let $M^\frakp(\udelta)$ be the parabolic Verma module in
  $\catO^\frakp(\gl_N)$ with highest weight $\udelta=-\delta(\epsilon_1+\cdots+\epsilon_m)$. Then the action from above
  factors through some cyclotomic quotient $\VB_A(\bomega;
  \beta_1,\beta_2; \beta_1^*, \beta_2^*)$, and we have an
  isomorphism of algebras
  \begin{equation}
    \label{eq:262}
    \VB_A(\boldsymbol \omega; \beta_1, \beta_2; \beta_1^*, \beta_2^*) \cong \End_{\gl_{n+m}}
    (M^\frakp(\udelta) \otimes V^{\otimes A}).
  \end{equation}
The parameters $\bomega, \beta_1,\beta_2,\beta_1^*,\beta_2^*$ depend explicitly on $\delta,m,n$.
\end{theorem*}

As a direct corollary of the second theorem, the cyclotomic quotient
$\VB_A(\bomega;\beta_1,\beta_2;\beta_1^*,\beta_2^*)$ inherits a
natural grading and the structure of a graded cellular algebra. Note
that the grading, however, depends on the chosen ordering of the
factors, that is on the $(r,t)$-sequence $A$.

Our definition of the algebra $\VB_{r,t}(\bomega)$ is inspired by
Nazarov's affine Wenzl algebra $\Wenzl_r(\bomega)$ \cite{MR1398116},
which can be thought as a degenerate affine version of the Brauer
algebra $\rmB_r(N)$. In particular, it is generated by the standard
generators of the walled Brauer algebra $\Br_{r,t}(\omega_0)$ together
with a polynomial ring $\C[y_1,\ldots,y_{r+t}]$.
Inspired from \cite{MR1398116} we use formal power series to handle the parameters $\omega_k$.
 
Our result has an analogue for types $B$, $C$ and $D$ of the form
\begin{equation}\label{eq:247}
  \frakg \acts \big( M \otimes V^{\otimes r} \big) \actsb \Wenzl_r(\bomega)
\end{equation}
established in \cite{ES2013}. Note however that
although $\Br_{r,t}(N)$ is a subalgebra of $\rmB_{r+t}(N)$, the
degenerate affine version $\VB_{r,t}(\bomega)$ is not (at least in a
natural way) a subalgebra of $\Wenzl_{r+t}(\bomega)$. This indicates that there is a close, but non-trivial relationship between type $A$ and type $B$, $C$, $D$ Lie algebras.

Using the diagrammatical description of category $\catO^\frakp$ from \cite{MR2781018} one can show that the KLR 2-category $\dot{\mathucal{U}}$ \cite{1206.17015} acts on $\catO^\frakp$ and moreover there is an isomorphism of algebras
\begin{equation}
  \label{eq:253}
  \End_{\dot{\mathucal{U}}}(\mathucal{F} \boldsymbol{1}_\lambda) / I \cong \VB_A (\bomega;\beta_1,\beta_2;\beta_1^*,\beta_2^*),
\end{equation}
where $\mathucal{F}$ is a $1$-morphism in $\dot{\mathucal{U}}$ corresponding to the chosen sequence $A$, the weight $\lambda$ is determined by the highest weight of the module $M^\frakp(\udelta)$ of \eqref{eq:262} and the ideal $I$ indicates that we have to take some cyclotomic quotient, determined also by the highest weight $\udelta$ or equivalently by the parameters $\beta_1,\beta_2,\beta_1^*,\beta_2^*$. This result will appear as part of joint work with Ehrig and Stroppel in \cite{miophd}. 

We point out that degenerate affine walled Brauer algebras have been defined independently by Rui and Su \cite{2013arXiv1305.0450R}. They prove another version of Schur-Weyl duality involving the Lie superalgebra $\gl(m|n)$ and their degenerate affine walled Brauer algebra. Our two approaches are connected via the super duality of \cite{1234.17004} established in \cite{MR2881300} relating $\gl(m|n)$ with the classical Lie algebra $\gl_{m+n}$.

\subsubsection*{Structure of the paper} We define the degenerate affine walled Brauer category and algebras in Section~\ref{sec:degen-affine-wall}. In Section~\ref{sec:graph-descr} we give a diagrammatical description that allows us to describe a set of generators as a vector space. In Section~\ref{sec:centre} we compute the centre of the degenerate affine walled Brauer category. In Section~\ref{sec:acti-gl_n-repr} we define the action on $\gl_N$-representations and state the first main theorem, whose proof will be presented as a series of lemmas in Section~\ref{sec:proofs}. In Section~\ref{sec:cyclotomic-quotients} we define cyclotomic quotients and prove our second main result. In Section~\ref{sec:diagr-part-calc} we will compute explicitly the generalized eigenvalues of the $y_i$'s in the cyclotomic quotients.

\subsubsection*{Acknowledgements} The present paper is part of the author's PhD thesis. The author would like to thank his advisor Catharina Stroppel for her help and support.

\section{The degenerate affine walled Brauer category}
\label{sec:degen-affine-wall}

In this section we define the degenerate affine walled Brauer category. We will indicate by $\bbS_n$ the symmetric group with simple reflections $\sfs_i$ for $i=1,\ldots,n-1$.

In the following by an $(r,t)$-sequence $A=(a_1,\ldots,a_{r+t})$ we will mean a permutation of the sequence
    \begin{equation}
      \label{eq:1}
      (\underbrace{1,\ldots,1}_{r},\underbrace{-1,\ldots,-1}_{t}).
    \end{equation}
We let $\Seq_{r,t}$ denote the set of $(r,t)$-sequences.

\begin{definition}\label{def:2}
  Let $r,t \in \N$ and fix a sequence $\boldsymbol
  \omega=(\omega_k)_{k \in \N}$ of complex parameters. The {\em
    degenerate affine walled Brauer category} $\uVB_{r,t}(\boldsymbol
  \omega)$ is the category defined as follows. The objects are
  $(r,t)$-sequences $A=(a_1,\ldots,a_{r+t}) \in \Seq_{r,t}$.
        Morphisms are generated by the following endomorphisms of $A$:
    \begin{equation}
      \label{eq:2}
      \begin{aligned}
        s^{(A)}_{i} & \qquad \text{for all } 1\leq i \leq r+t-1 \text{ such that } a_i=a_{i+1},\\ 
        e^{(A)}_{i} & \qquad \text{for all } 1\leq i \leq r+t-1 \text{ such
          that } a_i\neq a_{i+1},\\ 
        y^{(A)}_{i} & \qquad \text{for all } 1\leq i \leq r+t.
      \end{aligned}
    \end{equation}
    and the following two morphisms $A \rightarrow A'$
    \begin{equation}
      \label{eq:3}
        \hat{s}^{(A)}_{j}, \hat{e}^{(A)}_{j},
    \end{equation}
    where $A'=
    \mathsf s_j A$ for some simple reflection $\mathsf s_j \in \bbS_{r+t}$
    such that $A' \neq A$. We will often omit the superscript $A$.

    We impose the following relations on morphisms (where we use $\dot s_i,
    \dot e_i$ to denote both $s_i, \hat s_i$ and $e_i, \hat e_i$
    respectively; the relations are assumed to hold for all possible
    choices that make sense):
    \begin{enumerate}[(1),ref=\arabic*,itemsep=1ex]
    \item \label{item:1} $\dot s_i \dot s_i=1$,
    \item \label{item:2}
      \begin{enumerate}[(a),ref=\alph*]
      \item \label{item:3} $\dot s_i \dot s_j=\dot s_j \dot s_i$ for
        $\abs{i-j}>1$,
      \item \label{item:4} $\dot s_i \dot s_{i+1} \dot s_i = \dot
        s_{i+1} \dot s_{i} \dot s_{i+1}$,
      \item \label{item:5}$\dot s_i y_j= y_j \dot s_i$ for $j \neq i,i+1$,
      \end{enumerate}
    \item \label{item:6}$(e^{(A)}_{i})^2= \omega_0 e^{(A)}_{i}$,
      \item \label{item:7}$e^{(A)}_{1} y_1^k e^{(A)}_{1}= \omega_k e^{(A)}_{1}$ for $k \in \N$ if $a_1=1,a_{2}=-1$,
      \item \label{item:8}
        \begin{enumerate}[(a), ref=\alph*]
        \item \label{item:9}$\dot s_i \dot e_j=\dot e_j \dot s_i$ and
          $\dot e_i \dot e_j = \dot e_j \dot e_i$ for $\abs{i-j}>1$,
        \item \label{item:10}$\dot e_i y_j= y_j \dot e_i$ for $j \neq i,i+1$,
        \item \label{item:11}$y_i y_j=y_j y_i$,
        \end{enumerate}
      \item \label{item:12}
        \begin{enumerate}[(a), ref=\alph*]
        \item \label{item:13}$\hat{s}_i \dot e_i = \dot e_i = \dot e_i \hat{s}_i$,
        \item \label{item:14}$\dot s_i \dot e_{i+1} \dot e_i = \dot
          s_{i+1} \dot e_i$ and $\dot e_i \dot e_{i+1} \dot s_i =
          \dot e_i \dot s_{i+1}$,
        \item \label{item:15}$\dot e_{i+1} \dot e_i \dot s_{i+1} =
          \dot e_{i+1} \dot s_i$ and $\dot s_{i+1} \dot e_i \dot
          e_{i+1} = \dot s_i \dot e_{i+1}$,
        \item \label{item:16}$\dot e_{i+1} \dot e_i \dot e_{i+1} =
          \dot e_{i+1} $ and $\dot e_i \dot e_{i+1} \dot e_i = \dot
          e_i$,
        \end{enumerate}
      \item \label{item:17}
        \begin{enumerate}[(a), ref=\alph*]
        \item \label{item:18}$s_i y_i - y_{i+1}s_i = - 1$ and $ s_i
          y_{i+1} -y_i s_i = 1$,
        \item \label{item:19}$\hat s_i y_i - y_{i+1} \hat s_i = \hat e_i$ and
          $\hat s_i y_{i+1} - y_i \hat s_i  = - \hat e_i$,
        \end{enumerate}
      \item \label{item:20}
        \begin{enumerate}[(a), ref=\alph*]
        \item \label{item:21}$\dot e_i(y_i + y_{i+1})=0$,
        \item \label{item:22}$(y_i + y_{i+1}) \dot e_i = 0$.
        \end{enumerate}
    \end{enumerate}
\end{definition}

\begin{remark}
  \label{rem:1}
  Notice that the relation (\ref{item:8}\ref{item:10}) for $\hat e_i$ is implied by the relations (\ref{item:17}\ref{item:19}), (\ref{item:8}\ref{item:11}) and (\ref{item:2}\ref{item:5}). Moreover, the relations (\ref{item:20}\ref{item:21}-\ref{item:22}) for $\hat e_i$ are implied by the same relations for $e_i$ together with (\ref{item:12}\ref{item:13}) and the invertibility of $\hat s_i$.
\end{remark}

\begin{remark}
  \label{rem:3}
  Notice that we need to impose the relation (\ref{item:6}) only for
  $A,i$ such that $a_i=1,a_{i+1}=-1$. In fact, consider $A'=\mathsf
  s_i A$; we have $(e^{(A')}_{i})^2 = e^{(A')}_{i} e^{(A')}_{i}= \hat
  s_i e^{(A)}_{i} \hat s_i \hat s_i e^{(A)}_{i} \hat s_{i} = \hat s_i
  (e^{(A)}_{i})^2 \hat s_i = \omega_0 \hat s_i e^{(A)}_{i} \hat s_i =
  \omega_0 e^{(A')}_{i}$. In the same way we obtain more generally
  $\dot e_i \dot e_i = \omega_0 \dot e_i$ for all possible choices.
\end{remark}

\begin{remark}
  \label{rem:7}
  One could give a more general definition taking the $\omega_k$'s to be formal central parameters. Definition~\ref{def:2} would then be a specialized version.
\end{remark}

We stress again that we will try to omit the superscript $A$ of the generators of $\uVB_{r,t}(\bomega)$ as often as possible. The formulas we write are then supposed to write for all choices which make sense.

\begin{definition}
  \label{def:1}
  Let $A \in \Seq_{r,t}$ and fix a sequence $\boldsymbol
  \omega$ of complex parameters. The {\em degenerate affine walled
    Brauer algebra} corresponding to $A$ is
  \begin{equation}
    \label{eq:4}
    \VB_{A}(\boldsymbol \omega) = \End_{\uVB_{r,t}(\boldsymbol \omega)}(A).
  \end{equation}
\end{definition}

\begin{remark}
  \label{rem:2}
  Notice that $\VB_A(\bomega)$ is generated by the $s_i^{(A)},
  e_i^{(A)}, y_i^{(A)}$ (for all $i$ such that these elements
  exist). However, these generators satisfy some non-trivial relations
  that are implied by the relations defining the whole category (that
  involve also other homomorphisms spaces).  Of course, we could
  try to define the degenerate affine walled Brauer algebra without using
  Definition \ref{def:2}. But then the relations, as the proofs that
  will follow, would be much more complicated.
\end{remark}

The following result is straightforward:

\begin{lemma}
  \label{lem:16}
  All degenerate affine walled Brauer algebras corresponding to $(r,t)$-sequences are isomorphic.
\end{lemma}

\begin{proof}
  The isomorphisms are given by multiplication with a finite composition of $\dot s_i$'s.
\end{proof}

Notice that all relations defining the degenerate affine walled Brauer category are symmetric in $i$ except for \eqref{item:7}. We will spend the rest of this section to compute $e_i^{(A)} y_i^l e_i^{(A)}$ for all indices $i$.
The following lemma generalizes the relations (\ref{item:17}\ref{item:18}-\ref{item:19}) to higher powers of $y_i$:

\begin{lemma}
  \label{lem:15}
  For all $k \in \N$ we have
  \begin{align}
    s_i y_i^k - y_{i+1}^k s_i & = - \sum_{\ell=1}^k y_{i+1}^{\ell-1} y_i^{k-\ell}, \label{eq:55}\\ \displaybreak[2]
    s_i y_{i+1}^k - y_{i}^k s_i & = \sum_{\ell=1}^k y_{i+1}^{\ell-1} y_i^{k-\ell}, \label{eq:155}\\ \displaybreak[2]
    \hat s_i y_i^k - y_{i+1}^k \hat s_i & =  \sum_{\ell=1}^k y_{i+1}^{\ell-1} \hat e_i y_i^{k-\ell},\label{eq:53}\\ \displaybreak[2]
    \hat s_i y_{i+1}^k - y_{i}^k \hat s_i & =  - \sum_{\ell=1}^k y_{i}^{\ell-1} \hat e_i y_{i+1}^{k-\ell}.\label{eq:156}
  \end{align}
\end{lemma}

\begin{proof}
  Argue by induction, as in \cite[Lemma 2.3]{MR2235339}.
\end{proof}

Following \cite{MR1398116}, we introduce formal power series. Let $u$ be a formal variable, and let us work with formal Laurent power series in $u^{-1}$. Then from \eqref{eq:55} we have
\begin{equation}
  \label{eq:54}
  \begin{split}
    s_i \frac{1}{u- y_i}& = u^{-1} s_i \bigg( \sum_{k=0}^\infty y_i^k u^{-k} \bigg)\\
    &=u^{-1} \sum_{k=0}^\infty \bigg( y_{i+1}^k s_i - \sum_{\ell=1}^k y_{i+1}^{\ell-1} y_i^{k-\ell} \bigg) u^{-k}\\
    &= \frac{1}{u-y_{i+1}} s_i - \frac{1}{(u-y_{i+1})(u-y_{i})}.
  \end{split}
\end{equation}
Similarly, from \eqref{eq:155}, \eqref{eq:53} and \eqref{eq:156} we get respectively:
  \begin{align}
    s_i \frac{1}{u-y_{i+1}}& = \frac{1}{u-y_i} s_i +
    \frac{1}{(u-y_{i+1})(u-y_i)}, \label{eq:81}\\
    \hat s_i \frac{1}{u-y_i} &= \frac{1}{u-y_{i+1}} \hat s_i +
    \frac{1}{u-y_{i+1}} \hat e_i \frac{1}{u-y_i}, \label{eq:56}\\
    \hat s_i \frac{1}{u-y_{i+1}} &= \frac{1}{u-y_i} \hat s_i -
    \frac{1}{u-y_{i}} \hat e_i \frac{1}{u-y_{i+1}}\label{eq:82}.
  \end{align}
  
Set now 
\begin{equation}
W_1(u)=\sum_{k=0}^\infty \omega_k u^{-k}.\label{eq:43}
\end{equation}
Choose $A \in \Seq_{r,t}$ with $a_1=1$, $a_2=-1$. Then relation (\ref{item:7}) can be rewritten in the compact form
\begin{equation}
\begin{aligned}
  e^{(A)}_{1} \frac{1}{u-y_1} e^{(A)}_{1} & = u^{-1} e^{(A)}_{1} \bigg(\sum_{k=0}^\infty y_1^k u^{-k}\bigg) e^{(A)}_{1}\\
  &= u^{-1} \bigg(\sum_{k=0}^\infty \omega_k u^{-k}\bigg) e^{(A)}_{1} = \frac{W_1(u)}{u} e^{(A)}_{1}.
\end{aligned}\label{eq:39}
\end{equation}

We remark that $\frac{W_1(u)}{u} = u^{-1} W_1(u) \in \C[[u^{-1}]]$.

\begin{lemma}
  \label{lem:18}
  Let $A' \in \Seq_{r,t}$ be such that $a'_1=-1$, $a'_2=1$. Then
  \begin{equation}
    \label{eq:57}
    e^{(A')}_{1} \frac{1}{u-y_1} e^{(A')}_{1} = \frac{W_1^*(u)}{u} e^{(A')}_{1}
  \end{equation}
  where
  \begin{equation}
    \label{eq:58}
    \frac{W_1^*(u)}{u}  =\frac{ W_1(-u) }{ u - W_1(-u)}.
  \end{equation}
\end{lemma}

\begin{proof}
Let $A=\mathsf s_1 A'$ and compute using \eqref{eq:56}:
\begin{equation}
  \label{eq:46}
  \begin{split}
    &e^{(A')}_{1} \frac{1}{u-y_1} e^{(A')}_{1}  = \hat s_1 e^{(A)}_{1} \hat s_1 \frac{1}{u-y_1} \hat s_1 e^{(A)}_{1} \hat s_1 \\
    &\qquad = \hat s_1 e^{(A)}_{1} \frac{1}{u-y_{2}} \hat s_1 \hat s_1 e^{(A)}_{1} \hat s_1 + \hat s_1 e^{(A)}_{1} \frac{1}{u-y_{2}} \hat e^{(A')}_{1} \frac{1}{u-y_1} \hat s_1 e^{(A)}_{1} \hat s_1\\
    &\qquad = \hat s_1 e^{(A)}_{1} \frac{1}{u+y_1} e^{(A)}_{1} \hat s_1 + \hat s_1 e^{(A)}_{1} \frac{1}{u+y_1} e^{(A)}_{1} \hat s_1 \frac{1}{u-y_1} \hat s_{1} e^{(A)}_{1} \hat s_1\\
    &\qquad = \frac{W_1(-u)}{u} \hat s_1 e^{(A)}_{1} \hat s_1 + \frac{W_1(-u)}{u} \hat s_1 e^{(A)}_{1} \hat s_1 \frac{1}{u-y_1} \hat s_1 e^{(A)}_{1} \hat s_1\\
    &\qquad = \frac{W_1(-u)}{u} e^{(A')}_{1} + \frac{W_1(-u)}{u} e^{(A')}_{1} \frac{1}{u-y_1} e^{(A')}_{1}.
  \end{split}
\end{equation}
The claim follows.
\end{proof}

\begin{remark}
  Notice that the map $* : \C[[u^{-1}]] \mapto \C[[u^{-1}]]$ defined
  by
  \begin{equation}
    \label{eq:254}
    f^*(u) = \frac{ f(-u)}{ 1- u^{-1} f(-u)}
  \end{equation}
  is an involution, that is $f^{**}(u) = f(u) $ for all $u \in
  \C[[u^{-1}]]$.\label{rem:5}
\end{remark}

Let $R=\C[y_1,\ldots,y_{r+t}]$. Then the same proof of Lemma \ref{lem:18} gives the following more general result:
\begin{lemma}
  \label{lem:29}
  Let $A \in \Seq_{r,t}$ and suppose $a_i \neq a_{i+1}$ for some index $i$. Suppose that for  $A'=\sfs_{i} A$ the following holds:
  \begin{equation}
    \label{eq:163}
    e_{i}^{(A')} \frac{1}{u-y_{i}} e_{i}^{(A')} =
    \frac{W^{(A')}_{i}(u)}{u} e_{i}^{(A')},
  \end{equation}
  for some $W^{(A')}_i \in R[[u^{-1}]]$. Then
  \begin{equation}
    \label{eq:164}
    e_{i}^{(A)} \frac{1}{u-y_{i}} e_{i}^{(A)} =
    \frac{W^{(A)}_{i}(u)}{u} e_{i}^{(A)},
  \end{equation}
  where
  \begin{equation}
    \label{eq:165}
    \frac{W^{(A)}_i (u)}{u} = \frac{W^{(A')}_i(-u)}{u-W^{(A')}_i(-u)}.
  \end{equation}
\end{lemma}

\begin{lemma}
  \label{lem:28}
  Let $A \in \Seq_{r,t}$ with $(a_i,a_{i+1},a_{i+2})=(1,1,-1)$ for some index $i$. Suppose that for $A'=\sfs_{i+1} A$ the following holds:
  \begin{equation}
    \label{eq:162}
    e_{i}^{(A')} \frac{1}{u-y_{i}} e_{i}^{(A')} =
    \frac{W^{(A')}_{i}(u)}{u} e_{i}^{(A')},
  \end{equation}
  for some $W^{(A')}_i \in R[[u^{-1}]]$. Then
  \begin{equation}
    \label{eq:79}
    e_{i+1}^{(A)} \frac{1}{u-y_{i+1}} e_{i+1}^{(A)} =
    \frac{W^{(A)}_{i+1}(u)}{u} e_{i+1}^{(A)},
  \end{equation}
  where $W^{(A)}_{i+1}(u)$ is determined by
  \begin{equation}
    \label{eq:89}
    \frac{W^{(A)}_{i+1}(u)+u}{W^{(A')}_{i}(u)+u} = \frac{(u-y_i)^2}{(u-y_i)^2-1}.
  \end{equation}
\end{lemma}

\begin{proof}
  The proof is a direct calculation. First, we use $s_i^2=1$ and \eqref{eq:81}:
  \begin{equation}
    \label{eq:80}
    \begin{aligned}
      &e_{i+1} \frac{1}{u-y_{i+1}}e_{i+1}  =  e_{i+1} s_i s_i
      \frac{1}{u-y_{i+1}}e_{i+1}\\
      &\qquad= e_{i+1} s_i \frac{1}{u-y_i} s_i e_{i+1} + e_{i+1} s_i
      \frac{1}{(u-y_{i+1})(u-y_i)} e_{i+1}.
    \end{aligned}
  \end{equation}
  Since $e_{i+1}=s_i \hat s_{i+1} e_i \hat s_{i+1} s_i$, and since
  $y_i$ commutes with $\hat s_{i+1}$, we can rewrite the first summand
  as
  \begin{equation}
    \label{eq:84}
    s_i \hat s_{i+1} e_i \frac{1}{u-y_i} e_i \hat s_{i+1} s_i =
    \frac{W_i (u)}{u} e_{i+1}.
  \end{equation}
  For the second summand of \eqref{eq:80}, since
  $y_{i+1}e_{i+1}=-y_{i+2}e_{i+1}$ and $y_{i+2}$ commutes with $s_i$,
  we can write
  \begin{equation}
    \label{eq:85}
    e_{i+1}s_i  \frac{1}{(u-y_{i+1})(u-y_i)} e_{i+1} = e_{i+1}
    \frac{1}{u-y_{i+1}} s_i \frac{1}{u-y_i} e_{i+1}
  \end{equation}
  and using \eqref{eq:54} we get
  \begin{equation}
    \label{eq:86}
    e_{i+1} s_i \frac{1}{(u-y_i)^2} e_{i+1} +
    e_{i+1}\frac{1}{(u-y_{i+1})(u-y_i)^2} e_{i+1}, 
  \end{equation}
  that by the commutativity of $y_i$ and $e_{i+1}$ is
  \begin{equation}
    \label{eq:87}
    \frac{1}{(u-y_i)^2} e_{i+1} + \frac{1}{(u-y_i)^2} e_{i+1}
    \frac{1}{u-y_{i+1}} e_{i+1}. 
  \end{equation}
  Putting all together, we obtain
  \begin{equation}
    \label{eq:88}
    \left( 1- \frac{1}{(u-y_i)^2} \right) e_{i+1} \frac{1}{u-y_{i+1}}
    e_{i+1} = \left( \frac{W_i(u)}{u} + \frac{1}{(u-y_i)^2} \right) e_{i+1},
  \end{equation}
  that is equivalent to our claim.
\end{proof}

We have then:

\begin{prop}
  \label{prop:7}
  Let $A \in \Seq_{r,t}$ and let $i$ be an index such that $a_i \neq a_{i+1}$. Then we have
  \begin{equation}
    \label{eq:255}
    e_{i}^{(A)} \frac{1}{u-y_{i}} e_{i}^{(A)} =
    \frac{W^{(A)}_{i}(u)}{u} e_{i}^{(A)},
  \end{equation}
  for some power series $W^{(A)}_i(u) \in \C[y_1,\ldots,y_{i-1}][[u^{-1}]]$. Moreover, the power series $W^{(A)}_i(u)$ depends only on $(a_1,\ldots,a_i)$, that is $W^{(A)}_i = W^{(A')}_i$ if the sequences $A$ and $A'$ coincide up to the index $i$.
\end{prop}

\begin{proof}
  We prove by induction on $i$ that \eqref{eq:255} holds for all $(r,t)$-sequences such that $a_i \neq a_{i+1}$. If $i=1$ then this follows from \eqref{eq:39} or \eqref{eq:57}. Now let us suppose that the claim holds for $i$ and consider $i+1$. If $a_i = a_{i+1}$ then we can apply Lemma \ref{lem:28}. Otherwise we can apply Lemma \ref{lem:28} and get the claim for $\sfs_{i+1} A$, and then deduce the result for $A$ using Lemma \ref{lem:29}.
\end{proof}

\section{A diagrammatical description}
\label{sec:graph-descr}

We give now a graphical description of the degenerate affine walled Brauer category, that we will use to describe a set of generators as a vector space.

To $A \in \Seq_{r,t}$ we assign a sequence of $r+t$ oriented points on a horizontal line, numbered from $1$ to $r+t$ from left to right. Each point can be oriented upwards or downwards: the point $i$ is oriented upwards if $a_i=1$ and downwards otherwise.

Given $A,A' \in \Seq_{r,t}$, a morphism $\phi \in
\Hom_{\uVB_{r,t}(\bomega)}(A,A')$ is a $\C$-linear combination of
strand diagrams that connect the point sequence corresponding to $A$
to the point sequence corresponding to $A'$. In each strand diagram
the strands are oriented according to the orientations of the
endpoints. The generating morphisms are:
\begin{gather}\label{eq:62}
  s^{(A)}_{i} = \quad
  \begin{tikzpicture}[xscale=0.7,baseline=(current bounding box.center)]
    \draw (1,0) node[below] {$a_1$} -- ++(0,1) node[above] {$a_1$};
    \node at (2,0.5) {$\cdots$};
    \draw (3,0) node[below] {$a_{i-1}$} -- ++(0,1) node[above] {$a_{i-1}$};
    \draw[->] (4,0) node[below] {$a_{i}$} -- ++(1,1) node[above] {$a_{i+1}$};
    \draw[->] (5,0) node[below] {$a_{i+1}$} -- ++(-1,1) node[above] {$a_{i}$};
    \draw (6,0) node[below] {$a_{i+2}$} -- ++(0,1) node[above] {$a_{i+2}$};
    \node at (7,0.5) {$\cdots$};
    \draw (8,0) node[below] {$a_{r+t}$} -- ++(0,1) node[above] {$a_{r+t}$} ;
  \end{tikzpicture}\\ \displaybreak[2]
  e^{(A)}_{i} = \quad
  \begin{tikzpicture}[xscale=0.7,baseline=(current bounding box.center)]
    \draw (1,0) node[below] {$a_1$} -- ++(0,1) node[above] {$a_1$};
    \node at (2,0.5) {$\cdots$};
    \draw (3,0) node[below] {$a_{i-1}$} -- ++(0,1) node[above] {$a_{i-1}$};
    \draw[->] (4,0) node[below] {$a_{i}$} arc(180:0:0.5 and 0.35) node[below] {$a_{i+1}$};
    \draw[<-] (4,1) node[above] {$a_{i}$} arc(180:360:0.5 and 0.35) node[above] {$a_{i+1}$};
    \draw (6,0) node[below] {$a_{i+2}$} -- ++(0,1) node[above] {$a_{i+2}$};
    \node at (7,0.5) {$\cdots$};
    \draw (8,0) node[below] {$a_{r+t}$} -- ++(0,1) node[above] {$a_{r+t}$} ;
  \end{tikzpicture}\\ \displaybreak[2]
  \hat s^{(A)}_{i} = \quad
  \begin{tikzpicture}[xscale=0.7,baseline=(current bounding box.center)]
    \draw (1,0) node[below] {$a_1$} -- ++(0,1) node[above] {$a_1$};
    \node at (2,0.5) {$\cdots$};
    \draw (3,0) node[below] {$a_{i-1}$} -- ++(0,1) node[above] {$a_{i-1}$};
    \draw[->] (4,0) node[below] {$a_{i}$} -- ++(1,1) node[above] {$a_{i}$};
    \draw[<-] (5,0) node[below] {$a_{i+1}$} -- ++(-1,1) node[above] {$a_{i+1}$};
    \draw (6,0) node[below] {$a_{i+2}$} -- ++(0,1) node[above] {$a_{i+2}$};
    \node at (7,0.5) {$\cdots$};
    \draw (8,0) node[below] {$a_{r+t}$} -- ++(0,1) node[above] {$a_{r+t}$} ;
  \end{tikzpicture}\\ \displaybreak[2]
  \hat e^{(A)}_{i} = \quad
  \begin{tikzpicture}[xscale=0.7,baseline=(current bounding box.center)]
    \draw (1,0) node[below] {$a_1$} -- ++(0,1) node[above] {$a_1$};
    \node at (2,0.5) {$\cdots$};
    \draw (3,0) node[below] {$a_{i-1}$} -- ++(0,1) node[above] {$a_{i-1}$};
    \draw[->] (4,0) node[below] {$a_{i}$} arc(180:0:0.5 and 0.35) node[below] {$a_{i+1}$};
    \draw[->] (4,1) node[above] {$a_{i+1}$} arc(180:360:0.5 and 0.35) node[above] {$a_{i}$};
    \draw (6,0) node[below] {$a_{i+2}$} -- ++(0,1) node[above] {$a_{i+2}$};
    \node at (7,0.5) {$\cdots$};
    \draw (8,0) node[below] {$a_{r+t}$} -- ++(0,1) node[above] {$a_{r+t}$} ;
  \end{tikzpicture}\\
  y^{(A)}_i = \quad
  \begin{tikzpicture}[xscale=0.7,baseline=(current bounding box.center)]
    \draw (1,0) node[below] {$a_1$} -- ++(0,1) node[above] {$a_1$};
    \node at (2,0.5) {$\cdots$};
    \draw (3,0) node[below] {$a_{i-1}$} -- ++(0,1) node[above] {$a_{i-1}$};
    \draw (4,0) node[below] {$a_{i}$} -- node[circle,fill,inner sep=2pt] {} ++(0,1) node[above] {$a_{i}$};
    \draw (5,0) node[below] {$a_{i+1}$} -- ++(0,1) node[above] {$a_{i+1}$};
    \draw (6,0) node[below] {$a_{i+2}$} -- ++(0,1) node[above] {$a_{i+2}$};
    \node at (7,0.5) {$\cdots$};
    \draw (8,0) node[below] {$a_{r+t}$} -- ++(0,1) node[above] {$a_{r+t}$} ;
  \end{tikzpicture}
\end{gather}
We did not draw the orientations of all the strands, but as we said
they are supposed to be oriented according to the
$(r,t)$-sequences. Moreover, for each of the first four generators we
have only depicted the case $a_i=1$; in the case $a_i=-1$ the
orientations are swapped.

Composition of morphisms is obtained by stacking
diagram vertically, from the bottom to the top. We will say that a strand is \emph{decorated} if there is at least one dot on it.

Let us now translate the defining relations of $\uVB_{r,t}(\bomega)$
into diagrams. Relations (\ref{item:1}),
(\ref{item:2}\ref{item:3}-\ref{item:4}), (\ref{item:8}\ref{item:9})
and (\ref{item:12}\ref{item:13}-\ref{item:16}) allow us exactly to
stretch undecorated strands. By relations (\ref{item:2}\ref{item:5})
and (\ref{item:8}\ref{item:10}-\ref{item:11}), dots can be moved
vertically as long as they do not step over crossings; by relations
(\ref{item:20}\ref{item:21}-\ref{item:22}), if they step over a
maximum or minimum, a sign appears. By relation (\ref{item:6}) we can
remove each closed undecorated circle and multiply by
$\omega_0$. Using Proposition \ref{prop:7} we can actually remove
every decorated circle on the cost of multiplying by the corresponding
term of the power series $W^{(A)}_{i}(u)$. The only relations we have
not yet considered are (\ref{item:17}\ref{item:18}-\ref{item:19}),
which we can use to make dots step over crossings.

By this discussion, it follows that we have a natural inclusion of the
\emph{walled Brauer category} $\uBr_{r,t}(\omega_0)$ of parameter
$\omega_0$ into $\uVB_{r,t}(\boldsymbol \omega)$. Here by an inclusion we mean that we have an inclusion of algebras for each $A \in \Seq_{r,t}$:
\begin{equation}
  \label{eq:256}
  \Theta_A : \Br_A(\omega_0) \into \VB_A(\bomega)
\end{equation}
and moreover an inclusion of $(\Br_A(\omega_0),\Br_{A'}(\omega_0))$-bimodules for all $A, A' \in \Seq_{r,t}$:
\begin{equation}
  \label{eq:256}
  \Hom_{\uBr_{r,t}(\omega_0)}(A,A') \into \Hom_{\uVB_{r,t}(\bomega)}(A,A'),
\end{equation}
where the bimodule structure on the target space is induced by $\Theta_A$ and $\Theta_{A'}$.

We will call a diagram representing a morphism $A \to B$ an \emph{$AB$-Brauer diagram}, or an $A$-Brauer diagram if $A=B$. It is \emph{undecorated} if there are no dots on it, and it is actually an element of
the corresponding walled Brauer algebra.  An $AB$-Brauer diagram is
called a \emph{monomial} if it is of the type
\begin{equation}
y_1^{\gamma_1} \cdots
y_{r+t}^{\gamma_{r+t}} D y_1^{\eta_1} \cdots y_{r+t}^{\eta_{r+t}},\label{eq:263}
\end{equation}
where $D$ is an undecorated $AB$-Brauer diagram. We say that such a
monomial is \emph{regular} if $\gamma_i=0$ whenever the point $a_i$ on
the bottom of $D$ is the left endpoint of a horizontal arc, and
$\eta_i \neq 0$ implies that the point $b_i$ on the top of $D$ is the
left endpoint of a horizontal arc.

For all pairs $A, B \in \Seq_{r,t}$ we define a filtration
\begin{equation}
  \{0\}=V_{-1} \subseteq V_0 \subseteq V_1 \subseteq \cdots\label{eq:257}
\end{equation}
 on $\Hom_{\uVB_{r,t}(\bomega)}(A,B)$ by letting $V_i$ be the vector span of all strand diagrams with at most $i$ dots.

For the purpose of this paper, we say that a $\C$-linear category is a \emph{filtered category} if all homomorphism spaces are filtered and composition of morphism respects the filtration. Analogously we say that a $\C$-linear category is a \emph{graded category} if all homomorphism spaces are graded and composition of morphisms respects the grading. Given a filtered category, if we replace all homomorphism spaces with the associated graded spaces we get the associated graded category.

In particular, the filtrations \eqref{eq:257} on all homomorphism spaces turn the degenerate affine walled Brauer category $\uVB_{r,t}(\bomega)$ into a filtered category. The associated graded category $\sfG$ has the same generators and the same relations of $\uVB_{r,t}(\bomega)$ except for (\ref{item:17}\ref{item:18}-\ref{item:19}), that have to be replaced by
\begin{itemize}
\item[(\ref{item:17}')] $\dot s_i y_i = y_{i+1} \dot s_i$ and $\dot s_i y_{i+1} = y_i \dot s_i$.
\end{itemize}

Given a $\C$-linear category $\calC$ with a finite set of object $\Ob(\calC)$, we let $\Mor(\calC) = \bigoplus_{M,N \in \Ob(\calC)} \Hom_\calC(M,N)$. We say that a subset $S \subset \Mor(\calC)$ \emph{generates} $\calC$ if it generates $\Mor(\calC)$ as a vector space.

\begin{prop}
  \label{prop:3}
  The regular monomials generate $\uVB_{r,t}(\boldsymbol \omega)$.
\end{prop}

\begin{proof} Let $\sfG$ be the associated graded category. It is
  straightforward to see that regular monomials generate $\sfG$,
  because in $\sfG$ dots can step over crossings thanks to the
  relation (\ref{item:17}'). It follows that regular monomials also
  generate $\uVB_{r,t}(\boldsymbol \omega)$.
\end{proof}

\section{The centre}
\label{sec:centre}

We are now going to determine the centre of the degenerate affine walled Brauer category. Our computations are analogous to the ones in \cite{2011arXiv1105.4207D}.

Let $R$ be the polynomial ring $R=\C[y_1,\ldots,y_{r+t}]$. Observe that for each pair $A,B \in \Seq_{r,t}$, the vector space $\Hom_{\uVB_{r,t}(\bomega)}(A,B)$ is an $R$-bimodule.

\begin{definition}[See also \cite{2011arXiv1105.4207D}]
  We say that a polynomial $p
  \in R$ satisfies \emph{$Q$-cancellation} with respect to the
  variables $y_1,y_2$ if
  \begin{equation}
    \label{eq:223}
    p(y_1,-y_1,y_3,\ldots,y_{rts}) = p(0,0,y_3,\ldots,y_{r+t}).
  \end{equation}
  Analogously we say that $p$ satisfies $Q$-cancellation with respect
  to the variables $y_i,y_j$ if $w\cdot p$ satisfies \eqref{eq:223},
  where $w \in \bbS_{r+t}$ is the permutation that exchanges $1$ with
  $i$ and $2$ with $j$ and $\bbS_{r+t}$ acts on $R$ permuting the
  variables.\label{def:8}
\end{definition}

We have then the following:
\begin{theorem}
  \label{thm:6}
  The centre of\/ $\uVB_{r,t}(\bomega)$ is isomorphic to the subring of $(\bbS_r \times \bbS_t)$-invariant polynomials $p \in R^{{\bbS_r \times \bbS_{t}}}$ which satisfy $Q$-cancellation with respect to the variables $y_r,y_{r+1}$. The isomorphism is given by
  \begin{equation}
    \label{eq:258}
    p \mapsto \sum_{A \in \Ob} (w_A \cdot p) \id_A
  \end{equation}
 where $\Ob$ is the set of objects of $\uVB_{r,t}(\bomega)$ (i.e.\ the set of $(r,t)$-sequences) and for each $A \in \Ob$ the element $w_A$ is a permutation such that $w_A \cdot (1^r, (-1)^t) = A$.
\end{theorem}

\begin{proof}
  Let $Z$ be the centre of $\uVB_{r,t}(\bomega)$. Recall that $Z$ is
  by definition $\End(\underline 1)$ where $\underline 1$ is the
  identity endofunctor of $\uVB_{r,t}(\bomega)$. Hence by definition
  an element $f \in Z$ is an element
  \begin{equation}
    f \in \bigoplus_{A \in
    \Ob} \End_{ \uVB_{r,t}(\bomega)} (A)\label{eq:264}
  \end{equation}
 such that $f \phi = \phi f$
  for all morphisms $\phi \in \Hom_{\uVB_{r,t}(\bomega)}(A,A')$ and
  for all pairs $A,A' \in \Seq_{r,t}$.

  Let us pick a $A \in \Seq_{r,t}$. As in the proof of \cite[Theorem
  4.2]{2011arXiv1105.4207D}, it is easy to show that for such an $f$
  to commute with all endomorphisms of $A$ the component of $f$ in
  $\End_{ \uVB_{r,t}(\bomega)} (A)$ has to be a polynomial $p_A \in
  R$. In particular, we must have $f = \sum_{A \in \Ob} p_A \id_A$.

Let us now fix $A_0=(1^r,(-1)^t)$. Since $p_{A_0}$ has to be central in $\VB_{A_0}(\bomega)= \End_{ \uVB_{r,t}(\bomega)} ({A_0})$, it follows from Lemmas \ref{lem:41} and \ref{lem:43} below that $p_{A_0} \in R^{\bbS_{r} \times \bbS_{t}}$ and $p_{A_0}$ has to satisfy $Q$-cancellation with respect to $y_r,y_{r+1}$. On the other side, it follows by the same lemmas and the fact that the $s_i$'s and $e_i$'s generate $\VB_{A_0}(\bomega)$ that such a $p_{A_0}$ is central in $\VB_{A_0}(\bomega)$.

Finally, it follows from Lemma \ref{lem:42} that such an
 $f$ is central in the whole category if and only if $p_A = w_A \cdot p_{A_0}$ for all $A = w_A \cdot A_0$.
\end{proof}

As a corollary of the theorem (and of its proof), we can describe the centre of the degenerate affine walled Brauer algebras:

\begin{corollary}
  \label{cor:4}
  Let $A \in \Seq_{r,t}$ and $W_A \subset \bbS_{r+t}$ be the subgroup that fixes $A$. Let also $i,j$ be two indices such that $a_i \neq a_j$. Then the centre of $\VB_A(\bomega)= \End_{\uVB_{r,t}(\bomega)}(A)$ consists of the polynomials $p \in R^{W_A}$ that satisfy $Q$-cancellation with respect to the variables $y_i,y_{j}$.
\end{corollary}

\begin{lemma}
  \label{lem:41}
  Let $A \in \Seq_{r,t}$ with $a_i=a_{i+1}$. The polynomial $p \in R$ commutes with $s_i$ in $\VB_A(\bomega)$ if and only if $\sfs_i \cdot p = p$.
\end{lemma}
\begin{proof}
  For notational convenience let us suppose $i=1$. Let $p \in R$ and write
  \begin{equation}
    p= \sum_{a,b \in \N} y_1^a y_2^b p_{a,b}\qquad\text{with }p_{a,b} \in \C[x_3,\ldots,x_{r+t}]. \label{eq:224}
  \end{equation}
  Using \eqref{eq:55} and \eqref{eq:155} we get
  \begin{equation}
    \label{eq:225}
    s_1 y_1^a y_2^b = y_2^a y_1^b s_1 - \frac{y_1^a y_2^b - y_2^a y_1^b}{y_1-y_2}
  \end{equation}
  and hence
  \begin{equation}
    \label{eq:226}
    s_1 p = (\sfs_1 \cdot p) s_1 - \frac{p - \sfs_1 \cdot p}{y_1 - y_2}. 
  \end{equation}
  It follows that $p$ commutes with $s_1$ if and only if $\sfs_1 \cdot p =
  p$.
\end{proof}

\begin{lemma}\label{lem:43}
  Let $A \in \Seq_{r,t}$ with $a_i \neq a_{i+1}$. The polynomial $p \in R$ commutes with $e_i$ in $\VB_A(\bomega)$ if and only if it satisfies $Q$-cancellation with respect to $y_i,y_{i+1}$.
\end{lemma}

\begin{proof}
  Let us suppose $i=1$ and write $p$ as in \eqref{eq:224}. We have
  \begin{equation}
    \label{eq:229}
    e_1p-pe_1=  e_1 \bigg(\sum_{\substack{a,b \in \N\\a+b>0}} y_1^a y_2^b p_{a,b}\bigg)- \bigg(\sum_{\substack{a,b \in \N\\a+b>0}} y_1^a y_2^b p_{a,b}\bigg) e_1.
  \end{equation}
  The elements $\{e_1 y_1^a y_2^b p_{a,b},y_1^a y_2^b p_{a,b} e_1
  \suchthat a,b \in \N, a+b>0\}$ are linearly independent. Hence
  $\eqref{eq:229}=0$ if and only if both summands of the r.h.s.\
  vanish. This happens exactly when $p$ satisfies $Q$-cancellation
  with respect to $y_1,y_2$.
\end{proof}

\begin{lemma}
  \label{lem:42}
  Let $A \in \Seq_{r,t}$ with $a_i \neq a_{i+1}$. Then there exists a polynomial $q \in R$ such that $\hat s_i p = q \hat s_i$ if and only if $p$ satisfies $Q$-cancellation with respect to $y_i,y_{i+1}$. In this case we have $q= \sfs_i \cdot p$.
\end{lemma}

\begin{proof}
  Again, let us suppose $i=1$, and let us write $p$ as in \eqref{eq:224}. Using \eqref{eq:53} and \eqref{eq:156} we get
  \begin{equation}
    \label{eq:227}
    \hats_1 y_1^a y_2^b = y_2^a y_1^b \hats_1 + (-1)^a \sum_{\ell=1}^{a+b} (-1)^\ell y_1^{a+b-\ell} \hate_1 y_1^{\ell -1}
  \end{equation}
  and hence
  \begin{equation}
    \label{eq:228}
    \hats_1 p = (\sfs_1 \cdot p) \hats_1 + \sum_{k \in \Z_{>0}} \left(\sum_{\ell =1}^{r} (-1)^\ell y_1^{r-\ell} \hate_1 y_1^{\ell -1} \right) \left( \sum_{b=0}^k (-1)^{k-b} p_{k-b,b}\right).
  \end{equation}
  Now the claim follows since $\sum_{b=0}^k (-1)^b
  p_{k-b,k}=0$ for every $k>0$ if and only if $p$ satisfies $Q$-cancellation with respect to $y_i,y_{i+1}$.
\end{proof}

\section{Action on $\gl_N$-representations}
\label{sec:acti-gl_n-repr}

In this section we will define an action of the degenerate affine walled Brauer category on $\gl_N$-representations.

Let us fix an integer $N$ and let $I=\{1,\ldots,N\}$. Let $\gl_N$ be
the Lie algebra of $N \times N$ matrices. We denote by $E_{ij}$ the
matrix that has a one at position $(i,j)$ and zeroes elsewhere.  We
let $\mathfrak h \subset \gl_N$ be the standard Cartan subalgebra of diagonal
matrices and $\gl_N = \frakn^+ \oplus \frakh \oplus \frakn^-$ the triangular decomposition of $\gl_n$. A basis of $\mathfrak h$ is given by the matrices $H_i=E_{ii}$. Let $\epsilon_i$ be the basis of $\mathfrak h^*$ dual
to $E_{ii}$.

The set of {\em roots} of $\gl_N$ is $\Pi=\{\epsilon_i-\epsilon_j
\suchthat i \neq j\}$; a root is positive if $i>j$, negative
otherwise. The set of positive (resp.\ negative) roots will be denoted by
$\Pi^+$ (resp.\ $\Pi^-$). The {\em simple roots} are $\alpha_i =
\epsilon_i - \epsilon_{i+1}$ for all $i=1,\ldots,N-1$. We will denote the
root vectors by $X_{ij}$ for all $i \neq j$ or $X_\alpha$ for $\alpha \in
\Pi$. 

On $\gl_N$ we consider the non-degenerate symmetric bilinear form
defined by
\begin{equation}
  \label{eq:5}
  (A|B)= \operatorname{tr} (AB).
\end{equation}
Notice that the set $\{H_i,X_\alpha \suchthat i \in I,
\alpha \in \Pi\}$ gives an orthonormal basis of $\gl_N$.

The {\em vector representation} of $\gl_N$ is the $N$-dimensional
vector space $V$ with basis $\{v_i \suchthat i \in I\}$ on which the
action of $\gl_N$ is given by
\begin{equation}
  \label{eq:7}
  X_{ij} v_k = \delta_{jk} v_i, \qquad H_i v_k = \delta_{ik} v_k,
\end{equation}
where $\delta_{ij}$ is the Kronecker delta.
The {\em dual vector representation} $V^*$ has basis $\{v^*_i
\suchthat i \in I\}$, and the action of $\gl_N$ is given explicitly by
\begin{equation}
  \label{eq:8}
  X_{ij} v^*_k = - \delta_{ik} v^*_j, \qquad H_i v^*_k = - \delta_{ik} v^*_k.
\end{equation}

We have linear homomorphisms $\sigma_{W,Z} : W \otimes Z \to Z \otimes
W$, $\tau_{W}: W \otimes W^* \to W \otimes W^*$, $\hat \tau_W : W \otimes
W^* \to W^* \otimes W$, where $W,Z$ are either $V$ or $V^*$ and
$V^{**}=V$, defined by
\begin{align}
    \sigma_{W,Z} :  w_i \otimes z_j & \longmapsto z_j \otimes
    w_i  \label{eq:10} \\
    \tau_{W}: w_i \otimes w^*_j & \longmapsto \delta_{ij} \sum_{k \in I}
    (w_k \otimes w_k^*) \label{eq:11}\\
    \hat \tau_{W} : w_i \otimes w^*_j & \longmapsto \delta_{ij} \sum_{k
      \in I} (w_k^* \otimes w_k) \label{eq:12}
\end{align}
where $w_i, z_i$ are either $v_i$ or $v^*_i$, and $v_i^{**}=v_i$. It is immediate to check that these are indeed homomorphisms of $\gl_N$-representations.

For a $(r,t)$-sequence $A$ we let $V^{\otimes A}=V^{a_1} \otimes
\cdots \otimes V^{a_{r+t}}$, where $V^1=V$ and $V^{-1}=V^*$. Let also
$M$ be a $\gl_N$-module. The linear homomorphisms \eqref{eq:10},
\eqref{eq:11} and \eqref{eq:12} induce the following endomorphisms of
$M \otimes V^{\otimes A}$ for all $1 \leq i,j \leq r+t-1$ such that
$a_i=a_{i+1}$, $a_{j} \neq a_{j+1}$:
\begin{align}
  s_i & = \id \otimes \id^{\otimes(i-1)} \otimes
  \sigma_{V^{a_i},V^{a_{i+1}}} \otimes
  \id^{\otimes(r+t-i-1)} \label{eq:14}\\
  e_j & = \id \otimes \id^{\otimes(i-1)} \otimes \tau_{V^{a_j}}
  \otimes \id^{\otimes(r+t-i-1)} \label{eq:15}
\end{align}
and the following homomorphisms $M \otimes V^{\otimes A} \to M \otimes
V^{\otimes A'}$ where $A'= \mathsf s_j A$ for some simple
transposition $\mathsf s_j \in \bbS_{r+t}$:
\begin{align}
  \hat s_i & = \id \otimes \id^{\otimes(i-1)} \otimes
  \sigma_{V^{a_i},V^{a_{i+1}}} \otimes
  \id^{\otimes(r+t-i-1)} \label{eq:13}\\
  \hat e_j & = \id \otimes \id^{\otimes(i-1)} \otimes
  \hat\tau_{V^{a_j}} \otimes \id^{\otimes(r+t-i-1)}. \label{eq:16}
\end{align}

Let $U(\gl_N)$ be the universal enveloping algebra of $\gl_N$. The
{\em Casimir operator} of $U(\gl_N) \otimes U(\gl_N)$ is
\begin{equation}
  \label{eq:6}
  \Omega = \sum_{i \in I} H_i \otimes H_i + \sum_{\alpha \in \Pi}
  X_\alpha \otimes X_{-\alpha}.
\end{equation}
The {\em Casimir element} of $U(\gl_N)$ is $C=m(\Omega)$, where
$m:U(\gl_N) \otimes U(\gl_N) \to U(\gl_N) $ is the multiplication.
Writing $\Omega=\sum x_{(1)} \otimes x_{(2)}$, we define for $0 \leq i < j \leq r+t$
\begin{equation}
  \label{eq:17}
  \Omega_{ij} = \sum 1 \otimes \cdots \otimes 1 \otimes x_{(1)} \otimes 1 \otimes \cdots \otimes 1 \otimes x_{(2)} \otimes 1 \otimes \cdots \otimes 1,
\end{equation}
where $x_{(1)}$ is at position $i$ and $x_{(2)}$ is at position $j$, starting with position $0$. Multiplication by $\Omega_{ij}$ defines an element of $\End_{\gl_N}(M \otimes V^{\otimes A})$, and we set for $1 \leq i \leq r+t$
\begin{equation}
  \label{eq:18}
  y_i = \sum_{0 \leq k < i} \Omega_{ki} + \frac{N}{2}.
\end{equation}

\begin{lemma}
  \label{lem:1}
  Let $M$ be a highest weight module. Let $A \in \Seq_{r,t}$ with
  $a_1=1,a_2=-1$. For all $k \in \N$ there exist $\omega_k(M) \in \C$
  with $\omega_0(M)=N$ such that $e_1 y_1^k e_1 = \omega_k(M)e_1$ as
  elements in $\End_{\gl_N}(M \otimes V^{\otimes A})$.
\end{lemma}

\begin{proof}
  Consider the composition
  \begin{equation}
    \label{eq:19}
    f: M = M \otimes \C \longrightarrow M \otimes V \otimes V^* \stackrel{y_1^k}{\longrightarrow} M \otimes V \otimes V^* \longrightarrow M \otimes \C = M,
  \end{equation}
  where the first map is the canonical inclusion and the last one is
  the evaluation. Since $M$ is a highest weight module, $f$ must be a
  multiple, say $\omega_k(M)$, of the identity. It is then clear that
  $e_1 y_1^k e_1 = \omega_k(M)e_1$ as elements in $\End_{\gl_N}(M
  \otimes V \otimes V^*)$. Adding identities on the following tensor
  factors, the identity holds also for $M \otimes V^{\otimes A}$ as in the
  statement.

  The fact that $\omega_0(M)=N$ follows by elementary direct
  computation, and is true for every module $M$.
\end{proof}

We are ready to state our first main theorem; the computations needed for the proof are collected in Section~\ref{sec:proofs}.

\begin{theorem}
  \label{thm:1}
  Let $M$ be a highest weight module for $\gl_N$, and let $\boldsymbol
  \omega=(\omega_k(M))_{k \in \N}$ be the sequence of complex numbers
  given by Lemma \ref{lem:1}. Then the assignment $A \mapsto M \otimes
  V^{\otimes A}$ and formulas \eqref{eq:14}, \eqref{eq:15}, \eqref{eq:13},
  \eqref{eq:16}, \eqref{eq:18} define a functor\/ $\VB_{r+t}(\boldsymbol
  \omega) \to \catO(\gl_N)$. In particular, we have a
  well-defined action of\/ $\VB_A(\boldsymbol \omega)$ on $M \otimes
  V^{\otimes A}$ for every $(r,t)$-sequence $A$.
\end{theorem}

\begin{proof}
  We need to show that the relations of the degenerate affine walled
  Brauer category are satisfied by definitions \eqref{eq:14},
  \eqref{eq:15}, \eqref{eq:13}, \eqref{eq:16}, \eqref{eq:18} for a
  highest weight module $M$. The relations are checked in details in
  Section~\ref{sec:proofs}.

  Relation (\ref{item:1}) is obvious, as are relations
  (\ref{item:2}\ref{item:3}-\ref{item:4}). Relation
  (\ref{item:2}\ref{item:5}) follows from Lemma \ref{lem:2}. Relation
  (\ref{item:6}) follows because $\omega_0(M)=N$. Relation
  (\ref{item:7}) is implicated by Lemma \ref{lem:1} and our choice of
  $\boldsymbol \omega$. Relation (\ref{item:8}\ref{item:9}) is
  trivial. Relation (\ref{item:8}\ref{item:10}) follows from Lemma
  \ref{lem:3} and Remark \ref{rem:1}. Relation
  (\ref{item:8}\ref{item:11}) is Lemma \ref{lem:6}.  Relation
  (\ref{item:12}\ref{item:13}) is straightforward, while relations
  (\ref{item:12}\ref{item:14}-\ref{item:16}) are shortly discussed as
  Lemma \ref{lem:7}. Relations (\ref{item:17}\ref{item:18}) are Lemma
  \ref{lem:8}, while relations (\ref{item:17}\ref{item:19}) are Lemma
  \ref{lem:11}.
\end{proof}

\section{Cyclotomic quotients}
\label{sec:cyclotomic-quotients}

We will consider in this section cyclotomic quotients of the degenerate affine walled Brauer category of level two.

Fix two positive integers $m,n \in \N$ and let $N=m+n$ and 
$I=\{1,\ldots,m+n\}$. Let $\gl_{m+n}$ be the complex general linear Lie algebra with its Levi subalgebra $\gl_m
\oplus \gl_n$; let $\mathfrak p = (\gl_m \oplus \gl_n) + \mathfrak
n^+$ be the corresponding standard parabolic subalgebra.

Let us set
\begin{equation}
  \label{eq:35}
  \rho = -\epsilon_2 -2 \epsilon_3- \cdots - (m+n-1) \epsilon_{m+n}.
\end{equation}

Let $\catO(m,n)=\catO_{\textit{int}}^\frakp(\gl_{m+n})$ be the
category of finitely generated $\gl_{n+m}$-modules that are locally
finite over $\mathfrak p$, semisimple over $\mathfrak h$, and have all
integral weights (when regarded as $\mathfrak{sl}_{m+n}
$-modules). This category is studied extensively in \cite{MR2781018}.
A full set of representatives for the isomorphism classes of
irreducible modules in $\catO(m,n)$ is given by the modules
$\{L(\lambda) \suchthat \lambda \in \Lambda(m,n)\}$, where
\begin{equation}
  \label{eq:36}
  \Lambda(m,n) = \left\{ \lambda \in \mathfrak h^* \, \left\vert\,
      \begin{aligned}
        &(\lambda + \rho, \epsilon_i -\epsilon_j) \in \Z \text{ for all } 1 \leq i,j \leq m+n,\\
        &(\lambda + \rho, \epsilon_1) > \cdots > (\lambda + \rho, \epsilon_m),\\
        &(\lambda + \rho, \epsilon_{m+1}) > \cdots > (\lambda + \rho,
        \epsilon_{m+n})
      \end{aligned}\right.
  \right\}
\end{equation}
and $L(\lambda)$ is the irreducible $\gl_{n+m}$-module of highest
weight $\lambda$. The module $L(\lambda)$ is the irreducible head
of the parabolic Verma module $M^\frakp(\lambda)$.
This parabolic
Verma module is also the largest quotient of the (non-parabolic) Verma
module $M(\lambda) \in \catO(\gl_{m+n})$ which lies in the parabolic subcategory $\catO(m,n)$.

Notice that the weights of the vector representation $V$ are $\epsilon_1,\ldots,\epsilon_{n+m}$, while  the weights
of $V^*$ are $-\epsilon_1,\ldots,-\epsilon_{n+m}$. By the tensor identity, the module $M(\lambda) \otimes V$ (resp.\ $M(\lambda) \otimes
V^*$) has a filtration with sections isomorphic to $M(\lambda +
\epsilon_j)$ (resp.\ $M(\lambda - \epsilon_j)$) for all $j \in I$. It
follows, by the characterization of the parabolic Verma modules and
because tensoring with $V$ and $V^*$ are endofunctors of $\catO(m,n)$ (see \cite[Lemma 4.3]{MR2781018}),
that $M^\frakp (\lambda) \otimes V$ has a filtration with sections
isomorphic to $M^\frakp(\lambda + \epsilon_j)$ for all $j$ such that
$\lambda+\epsilon_j \in \Lambda(m,n)$; similarly, $M^\frakp (\lambda)
\otimes V^*$ has a filtration with sections isomorphic to
$M^\frakp(\lambda -\epsilon_j)$ for all $j$ such that $\lambda -
\epsilon_j \in \Lambda(m,n)$.

For $\delta \in \Z$ we set
\begin{equation}
  \underline \delta= -\delta (\epsilon_1 +
  \cdots + \epsilon_m).\label{eq:37}
\end{equation}

\begin{lemma}
  \label{lem:13}
  Suppose $m,n\geq 1$ and $\delta \neq m$. Then there is an
  isomorphims of $\gl_{n+m }$-modules
  \begin{equation}
    M^\frakp(\udelta) \otimes V \cong M^\frakp(\udelta + \epsilon_1) \oplus M^\frakp(\udelta+\epsilon_{m+1}).\label{eq:9}
  \end{equation}
  This is also an eigenspace decomposition for the action of $y_1$,
  with eigenvalues
  \begin{equation}
    \label{eq:38}
    \beta_1 = -\delta + \frac{m+n}{2}\qquad \text{and}\qquad \beta_2 = \frac{n-m}{2}.
  \end{equation}
\end{lemma}

\begin{proof}
  By the discussion above and the definition of $\udelta$, we have
  that $M^\frakp(\udelta) \otimes V$ has a filtration with parabolic
  Verma modules $M^\frakp(\udelta+\epsilon_1)$ and
  $M^\frakp(\udelta+\epsilon_{m+1})$. Since $\udelta + \epsilon_1 > \udelta + \epsilon_{m+1}$, the term $M^\frakp(\udelta + \epsilon_1)$ is a submodule, hence we have
  \begin{equation}
    \label{eq:221}
    0 \to M^\frakp(\udelta+ \epsilon_1) \to M^\frakp(\udelta) \otimes V \to M^\frakp(\udelta+ \epsilon_{m+1}) \to 0.
  \end{equation}

  Let $C$ be the Casimir element of $\gl_{m+m}$ as defined in Section
  \ref{sec:acti-gl_n-repr}.  By a straightforward computation, $C$
  acts as $\langle \lambda, \lambda+2\rho\rangle$ on the highest
  vector of $M^\frakp(\lambda)$, and hence on the whole module. In
  particular, as $V$ is the irreducible head of
  $M^\frakp(\epsilon_1)$, $C$ act as $\langle \epsilon_1,
  \epsilon_1+2\rho\rangle$ on $V$. Note that $\Delta(C) = C \otimes 1
  + 1 \otimes C + 2\Omega$. Hence using the action of the Casimir
  element we can compute the action of $\Omega$ on $M^\frakp(\udelta +
  \epsilon_1)$, that is given by the scalar
  \begin{equation}
    \label{eq:20}
    \begin{split}
      & \frac{1}{2} \Bigl( \langle \udelta + \epsilon_1, \udelta + \epsilon_1 + 2\rho\rangle - \langle \udelta , \udelta + 2\rho\rangle - \langle \epsilon_1,\epsilon_1+2\rho\rangle\Bigr)\\
      & \quad
      \begin{multlined}
        = \frac{1}{2} \Bigl( \langle \udelta, \udelta + 2\rho \rangle
        + \langle\epsilon_1,\epsilon_1+2\rho\rangle +2\langle\udelta,
        \epsilon_1 \rangle -\\- \langle \udelta , \udelta + 2\rho\rangle
        - \langle \epsilon_1,\epsilon_1+2\rho\rangle\Bigr)
      \end{multlined}
\\
      & \quad = \langle \udelta, \epsilon_1\rangle = -\delta,
    \end{split}
  \end{equation}
  so that the action of $y_1$ is given by $-\delta + \frac{m+n}{2}$
  
  Analogously, the action of $\Omega$ on $(M^\frakp(\udelta) \otimes V )/M^\frakp(\udelta + \epsilon_1) \cong M^\frakp (\udelta +
  \epsilon_{m+1})$ is given by the scalar
  \begin{equation}
    \label{eq:23}
    \begin{split}
      & \frac{1}{2} \Bigl( \langle \udelta + \epsilon_{m+1}, \udelta + \epsilon_{m+1} + 2\rho\rangle - \langle \udelta , \udelta + 2\rho\rangle - \langle \epsilon_1,\epsilon_1+2\rho\rangle\Bigr)\\
      & \quad
      \begin{multlined}
        = \frac{1}{2} \Bigl( \langle \udelta, \udelta + 2\rho \rangle
        + \langle\epsilon_{m+1},\epsilon_{m+1}+2\rho\rangle
        +2\langle\udelta, \epsilon_{m+1} \rangle- \\ - \langle \udelta ,
        \udelta + 2\rho\rangle - \langle
        \epsilon_1,\epsilon_1+2\rho\rangle\Bigr)
      \end{multlined}
\\
      & \quad = \frac{1}{2} \Bigl( 1 +
      \langle\epsilon_{m+1},2\rho\rangle +2\langle \udelta,
      \epsilon_{m+1}\rangle -1 - \langle \epsilon_1,2\rho\rangle
      \Bigr)\\
      & \quad = \langle \epsilon_{m+1}-\epsilon_1 , \rho\rangle +
      \langle \udelta, \epsilon_{m+1}\rangle = -m ,
    \end{split}
  \end{equation}
  and the action of $y_1$ is given by $-m +\frac{m+n}{2}$.

  Since the two factors \eqref{eq:20} and \eqref{eq:23} are different, they are indeed eigenvalues for the action of $\Omega$ and the exact sequence \eqref{eq:221} splits. 
\end{proof}

\begin{remark}
  \label{rem:9}
  If $\delta=m$ then the two eigenvalues $\beta_1$ and $\beta_2$ coincide. In this case, the short exact sequence \eqref{eq:221} does not split. The element $(y_1 -\beta_1)^2$ vanishes on $M^\frakp(\udelta) \otimes V$.
\end{remark}

\begin{lemma}
  \label{lem:14}
  Suppose $m,n\geq 1$ and $\delta \neq n$. Then there is an
  isomorphims of $\gl_{n+m }$-modules
  \begin{equation}
    M^\frakp(\udelta) \otimes V^* \cong M^\frakp(\udelta-\epsilon_{m+n}) \oplus M^\frakp(\udelta - \epsilon_{m}) .\label{eq:69}
  \end{equation}
  This is also an eigenspace decomposition for the action of $y_1$,
  with eigenvalues
  \begin{equation}
    \label{eq:70}
    \beta^*_1 = \frac{m+n}{2} \qquad \text{and}\qquad \beta^*_2 =  \delta + \frac{m-n}{2}.
  \end{equation}
\end{lemma}

\begin{proof}
  The proof is analogous to the previous one. We just note that the
  highest weight of $V^*$ is $-\epsilon_{m+n}$ and compute the action
  of $\Omega$ on the summand $M^\frakp(\udelta-\epsilon_m)$:
  \begin{equation}
    \label{eq:44}
    \begin{split}
      & \frac{1}{2} \Bigl( \langle \udelta - \epsilon_m, \udelta - \epsilon_{m} + 2\rho\rangle - \langle \udelta , \udelta + 2\rho\rangle - \langle -\epsilon_{m+n},-\epsilon_{m+n}+2\rho\rangle\Bigr)\\
      & \quad \begin{multlined}= \frac{1}{2} \Bigl( \langle \udelta, \udelta + 2\rho \rangle - \langle\epsilon_m,-\epsilon_m+2\rho\rangle -2\langle\udelta, \epsilon_m \rangle -\\- \langle \udelta , \udelta + 2\rho\rangle + \langle \epsilon_{m+n},-\epsilon_{m+n}+2\rho\rangle\Bigr)
      \end{multlined}
      \\
      & \quad = -\langle \epsilon_m - \epsilon_{m+n},\rho\rangle -
      \langle \udelta, \epsilon_m\rangle =- n + \delta,
    \end{split}
  \end{equation}
  so that the action of $y_1$ is given by $\delta + \frac{m-n}{2}$,
  and on the summand $M^\frakp(\udelta-\epsilon_{m+n})$:
  \begin{equation}
    \label{eq:45}
    \begin{split}
      &
      \begin{multlined}
        \frac{1}{2} \Bigl( \langle \udelta - \epsilon_{m+n}, \udelta -
        \epsilon_{m+n} + 2\rho\rangle - \langle \udelta , \udelta +
        2\rho\rangle -\\- \langle
        -\epsilon_{m+n},-\epsilon_{m+n}+2\rho\rangle\Bigr)
      \end{multlined}
\\
      & \quad \begin{multlined}
        = \frac{1}{2} \Bigl( \langle \udelta, \udelta + 2\rho \rangle - \langle\epsilon_{m+n},-\epsilon_{m+n}+2\rho\rangle -2\langle\udelta, \epsilon_{m+n} \rangle -\\- \langle \udelta , \udelta + 2\rho\rangle + \langle \epsilon_{m+n},-\epsilon_{m+n}+2\rho\rangle\Bigr)
      \end{multlined} \\
      & \quad = - \langle \udelta, \epsilon_{m+n}\rangle = 0,
    \end{split}
  \end{equation}
  and the action of $y_1$ is given by $\frac{m+n}{2}$.
\end{proof}

\begin{remark}
  \label{rem:8}
  Also in this case, if $\delta=n$ then $\beta_1^* = \beta_2^*$ and instead of \eqref{eq:69} we have a short exact sequence that does not split. The element $(y_1 - \beta_1^*)^2$ vanishes on $M^\frakp(\udelta) \otimes V^*$.
\end{remark}

We define now the cyclotomic walled Brauer category of level two:

\begin{definition}
  \label{def:4}
  Let $r,t \in \N$ and fix a sequence $\boldsymbol \omega$ of complex
  parameters. Let also $\beta_1,\beta_2,\beta_1^*,\beta_2^*$ be
  complex numbers with $\beta_1\neq\beta_2$ and
  $\beta_1^*\neq\beta_2^*$. The {\em cyclotomic walled Brauer
    category} $\uVB_{r,t}(\boldsymbol \omega;
  \beta_1,\beta_2;\beta_1^*,\beta_2^*)$ is obtained imposing to the
  degenerate affine walled Brauer category $\uVB_{r,t}(\bomega)$ the following relations:
  \begin{align}
    (y_{A,1} - \beta_1)(y_{A,1}-\beta_2) = 0 & \qquad \text{for every
    } A \in \Seq_{r,t}
    \text{ with } a_1=1, \label{eq:48}\\
    (y_{A',1} - \beta^*_1)(y_{A',1}-\beta^*_2) = 0 & \qquad \text{for
      every } A'\in \Seq_{r,t} \text{ with } a'_1=-1. \label{eq:49}
  \end{align}
  
  If $A$ is an $(r,t)$-sequence, we define the {\em cyclotomic walled
    Brauer algebra}
  \begin{equation}
    \label{eq:47}
    \VB_A(\boldsymbol \omega; \beta_1, \beta_2; \beta_1^*,\beta_2^*)
    = \End_{\uVB_{r,t}(\boldsymbol \omega; \beta_1, \beta_2; \beta_1^*,
      \beta_2^*)}(A).
  \end{equation}
\end{definition}

\begin{remark}
  We remark that we really need to quotient out both \eqref{eq:48} and
  \eqref{eq:49} in order to be sure that we get a finite dimensional
  quotient. Moreover, we point out that it is important to first take
  the cyclotomic quotient of the whole category and then define the
  cyclotomic walled Brauer algebras to be the endomorphism algebras in
  the cyclotomic category: if we would define the cyclotomic algebras
  to be the cyclotomic quotients of the degenerate affine algebras by
  relations \eqref{eq:48} or \eqref{eq:49}, then they would not be in
  general finite dimensional.\label{rem:6}
\end{remark}

For general parameters $\bomega, \beta_1,\beta_2, \beta_1^*,
\beta_2^*$ we can not say anything about the cyclotomic quotient $\uVB_{r,t}(\bomega; \beta_1,\beta_2; \beta_1^*, \beta_2^*)$, which could even be trivial. However, if the parameters are chosen carefully, we will prove that the cyclotomic walled Brauer algebras are finite dimensional of dimension $2^{r+t} (r+t)!$, as one would expect.

We have the following consequence of the definition and of Theorem \ref{thm:1}:

\begin{corollary}
  \label{cor:1} Fix integers $r,t$ with $r\geq 1$, and let $A \in \Seq_{r,t}$. Fix also $\delta \in \Z$ and $m,n \geq 1$. Then the action of Theorem
  \ref{thm:1} factors through the cyclotomic quotient, defining a
  functor
  \begin{equation}\label{eq:260}
    \begin{aligned}
      \uVB_{r,t}(\bomega; \beta_1, \beta_2; \beta_1^*,\beta_2^*) &
      \longrightarrow \catO(m,n)\\
      A & \longmapsto M^{\frakp}(\udelta) \otimes V^{\otimes A}
    \end{aligned}
\end{equation}
and in particular an action of\/ $\VB_A(\boldsymbol \omega; \beta_1,
\beta_2; \beta_1^*, \beta_2^*)$ on $M^{\frakp}(\udelta) \otimes
V^{\otimes A}$, where $\boldsymbol \omega$ is as in Theorem
\ref{thm:1}, while $\beta_1,\beta_2,\beta_1^*,\beta_2^*$ are given by
Lemmas \ref{lem:13} and \ref{lem:14}.
\end{corollary}

We will need a finite set of generators for the cyclotomic walled Brauer category:

\begin{prop}
  \label{prop:1}
  The regular monomials 
  \begin{equation}
y_1^{\gamma_1}\cdots y_{r+t}^{\gamma_{r+t}} B
  y_1^{\eta_1}\cdots y_{r+t}^{\eta_{r+t}} \qquad \text{with } \gamma_i,\eta_j \in
  \{0,1\} \text{ for all } i,j\label{eq:259}
\end{equation}
generate the cyclotomic walled Brauer
  category $\uVB_A(\boldsymbol \omega; \beta_1, \beta_2;\beta_1^*,\beta_2^*)$.
\end{prop}

\begin{proof}
  As in the proof of Proposition \ref{prop:3}, it is enough to prove
  the statement for the associated graded category $\sfG'$ (since the
  filtration on the degenerate affine walled Brauer category descends
  to a filtration on the cyclotomic quotient). Of course all regular
  monomials generate the cyclotomic quotient by Proposition
  \ref{prop:3}. Consider a regular monomial: we can move the
  strands so that every strand at some level happens to be the
  leftmost strand. Now if some $\gamma_i$ or $\eta_i$ is bigger than
  $1$, then there are at least two dots on some strand. Using relations (\ref{item:17}') for $\sfG'$, we can move the two dots along the strand until they reach the level at which there are no other strands on their left. By the graded cyclotomic relation, this monomial is zero in the cyclotomic quotient.
\end{proof}

We will call the elements \eqref{eq:259} \emph{cyclotomic regular monomials}. We are now ready to prove our second main result:

\begin{theorem}
  \label{thm:2}
  Let $m,n,r,t,\delta$ be integer numbers with $m,n,r \geq 1$, $t \geq 0$  and $m,n \geq r+t$. Let $\bomega=\bomega(M^\frakp(\udelta))$ be given by Lemma~\ref{lem:1} and
  \begin{align}
    \label{eq:265}
    \beta_1 & = - \delta + \frac{m+n}{2}, &\beta_2 &= \frac{n-m}{2},\\
    \beta_1^* &= \frac{m+n}{2}, & \beta_2^* &= \delta+ \frac{m-n}{2}
  \end{align}
  as given by Lemmas~\ref{lem:13} and \ref{lem:14}.

  Then the cyclotomic regular $A$-monomials of Proposition
  \ref{prop:1} form a basis of $\VB_A(\boldsymbol \omega;
  \beta_1,\beta_2; \beta_1^*,\beta_2^*)$ and we have an isomorphism
  of algebras
  \begin{equation}
    \label{eq:40}
    \VB_A(\boldsymbol \omega; \beta_1, \beta_2; \beta_1^*, \beta_2^*) \cong \End_{\gl_{n+m}}
    (M^\frakp(\udelta) \otimes V^{\otimes A}),
  \end{equation}
  In particular $\dim_\C \VB_A(\boldsymbol \omega; \beta_1, \beta_2; \beta_1^*, \beta_2^*) = 2^{r+t} (r+t)!$
\end{theorem}

Before proving the theorem, let us state the following important corollary:

\begin{corollary}
  \label{cor:2}
  With the hypotheses of Theorem \ref{thm:2}, the cyclotomic walled Brauer algebra $\VB_A(\boldsymbol \omega; \beta_1, \beta_2; \beta_1^*, \beta_2^*)$ inherits a grading and a graded cellular algebra structure, where the graded decomposition numbers are given by parabolic Kazhdan-Lusztig polynomials of type $A$.
\end{corollary}

\begin{proof}
  By Theorem \ref{thm:2} we need to prove the claim for the
  endomorphism algebra
  \begin{equation}
    \End_{\catO(m,n)}(M^\frakp(\udelta) \otimes V^{\otimes A}).\label{eq:22}
  \end{equation}
  Notice that the weight $\udelta$ is either a dominant weight of or
  an anti-dominant weight for the parabolic category $\catO(m,n)$ (it
  is dominant if $\delta \leq n$ and it is anti-dominant weight if
  $\delta \geq m$, but it could also happen that it is both dominant
  and anti-dominant). Hence the parabolic Verma module
  $M^\frakp(\udelta)$ is either a projective module or a tilting
  module in $\catO(m,n)$.

  If $M^\frakp(\udelta)$ is projective then $M^\frakp(\udelta) \otimes
  V^{\otimes A}$ is also a projective module. If $M^\frakp(\udelta)$
  is tilting, then $M^\frakp(\udelta) \otimes V^{\otimes A}$ is also
  tilting. Since the blocks of $\catO(m,n)$ are Ringel self-dual (for
  the regular blocks this is \cite[Proposition~4.4]{MR2369489}, and
  since $\frakp \subset \gl_N$ is a maximal parabolic subalgebra the
  singular blocks are equivalent to regular blocks for smaller $N$'s),
  the endomorphism algebra of this tilting module is isomorphic to the
  endomorphism algebra of a projective module.

  In both cases, \eqref{eq:22} is then isomorphic to an idempotent
  truncation of the endomorphism algebra of a projective generator of
  a sum of blocks of $\catO(m,n)$. Since blocks of $\catO(m,n)$ are graded
  quasi-hereditary, this idempotent truncation inherits the structure
  of a graded cellular algebra (see \cite[Proposition~4.3]{MR1648638}).
\end{proof}

\begin{proof}[Proof of Theorem \ref{thm:2}]
  First, let us compute the action of $y_1$ on $M^\frakp(\udelta)
  \otimes V$. We indicate with $z$ the highest vector of
  $M^\frakp(\udelta)$, and note that $\gl_m \oplus \gl_n \subset
  \frakp$ acts as $0$ on $z$, since $M^\frakp(\udelta) = U(\gl_{n+m})
  \otimes_{\frakp} E(\udelta)$ where $E(\udelta)$ is the simple $\gl_m
  \oplus \gl_n$-module with highest weight $\udelta$, that by our
  choice of $\udelta$ is one-dimensional. Hence $X_{ij} z=0$ whenever
  both $i,j \leq m$ or $i,j > m$, and we obtain:
  \begin{equation}
    \label{eq:41}
    y_1 (z \otimes v_i) =
    \begin{cases}
      \big(-\delta + \tfrac{m+n}{2}\big) z \otimes v_i
      & \text{if } i\leq m\\
      \big(\tfrac{m+n}{2}\big) z \otimes v_i + \displaystyle\sum_{k
        \in I, k\leq m} X_{ik} z \otimes v_k & \text{if } i> m.
    \end{cases}
  \end{equation}
  Analogously, let us compute the action of $y_1$ on
  $M^\frakp(\udelta) \otimes V^*$:
  \begin{equation}
    \label{eq:42}
    y_1 (z \otimes v^*_j) =
    \begin{cases}
      \big(\delta + \tfrac{m+n}{2}\big) z \otimes v^*_j -
      \displaystyle\sum_{k \in I, k>m} X_{kj} z \otimes v^*_k
      & \text{if } j\leq m\\
      \big(\tfrac{m+n}{2}\big) z \otimes v_i
      & \text{if } i> m.
    \end{cases}
  \end{equation}
  Now consider the action of $y_h$ on $M^\frakp(\udelta) \otimes
  V^{\otimes A}$. By Lemmas \ref{lem:9} and \ref{lem:10}, $y_h$ acts as
  $\Omega_{0h}$ plus some linear combination of $A$-walled Brauer
  diagrams.  The parabolic Verma $M^\frakp(\udelta)$ has a natural
  vector space grading gived by $\abs{X_{\zeta_1}\cdots X_{\zeta_\ell}
    z} = \ell$, where $\zeta_i$ are simple roots. This extends to a
  grading on $M^\frakp(\udelta) \otimes V^{\otimes A}$. Then we have
  \begin{equation}
    \label{eq:50}
    y_h(z \otimes \cdots  v^{\chi}_{i}  \cdots ) =
    \begin{cases}
      \displaystyle\sum_{k \in I, k\leq m} X_{ik}z \otimes \cdots v_k
      \cdots & \text{if } i>m, \chi=1,\\
      - \displaystyle\sum_{k \in I, k> m} X_{ki}z \otimes \cdots v_k^*
      \cdots & \text{if } i\leq m, \chi=-1,\\
      0 & \text{otherwise}
    \end{cases}
  \end{equation}
  up to terms of degree zero.

  Now consider a cyclotomic regular monomial, and draw it as a
  decorated $A$-walled Brauer diagram $\aleph$. Remember that we read
  diagrams from the bottom to the top. We are going to explain a way
  to label the endpoints of $\aleph$. We consider the oriented arcs of
  $\aleph$ as arrows, having a source and a target. Order in some way
  the sources of the arrows of $\aleph$, labeling them sequentially
  with numbers $m+1,\ldots,m+r+t$. Now for every undecorated arrow
  label the target with the same label as the source. For every
  decorated arrow such that the source is labeled with $p$, label the
  target with $p-m$. Let $\tau_i$ be the label of the target of the
  arrow with source $i-m$.

  In this way, we obtain sequences 
  $1\leq b_h,c_k \leq m+n$ for $h,k=1,\ldots,r+t$ respectively on the
  bottom and on the top of our diagram.

  Let now $v^{\otimes A}_b= v^{a_1}_{b_1} \otimes \cdots \otimes
  v^{a_{r+t}}_{b_{r+t}}$ where as usual $v_i^1=v_i$ and
  $v_i^{-1}=v_i^*$. Take a zero linear combination $\sum_\gimel
  f_\gimel \, \gimel =0 $ of cyclotomic regular monomials. Pick a
  diagram $\aleph$ with all arcs decorated and let $b=b(\aleph)$,
  $c=c(\aleph)$, $\tau=\tau(\aleph)$ be the sequences constructed as
  before. By \eqref{eq:50} and by our construction, we have
  \begin{equation}
    \label{eq:51}
    \bigg\langle \sum_{\gimel} f_\gimel \, \gimel\, (v^{\otimes A}_{b(\aleph)}) ,
    X_{1,\tau_1}\cdots X_{r+t,\tau_{r+t}} v^{\otimes A}_{c(\aleph)} \bigg\rangle
    = \pm f_\aleph,
  \end{equation}
  where we have fixed on $M\otimes V^{\otimes A}$ the standard scalar
  product, with respect to that the standard basis is
  orthonormal. Hence $f_\aleph=0$.

  Now pick a diagram $\aleph$ with all but one arcs decorated, let
  $b,c,\tau$ be the sequences as above. Then equation \eqref{eq:51}
  again holds, if we do not write the $X_i, \tau_i$ corresponding to
  the undecorated arrow. Proceeding in this way, we have that all
  coefficients $f_\gimel$ are zero, hence the representation is
  faithful, or in other words the map
  \begin{equation}
    \label{eq:52}
    \VB_A(\boldsymbol \omega; \beta_1, \beta_2; \beta_1^*, \beta_2^*) \longrightarrow \End_{\gl_{n+m}}
    (M^\frakp(\udelta) \otimes V^{\otimes A})
  \end{equation}
  is injective.

  To prove surjectivity, we use a dimension argument. On one side,
  note that there are $2^{r+t}(r+t)!$ cyclotomic regular monomial,
  hence
  \begin{equation}
    \dim \VB_A(\boldsymbol \omega; \beta_1, \beta_2; \beta_1^*, \beta_2^*) \leq 2^{r+t}(r+t)!\label{eq:222}
\end{equation}
By the injectivity of \eqref{eq:52}, this is actually an equality. On the other side, by adjunction we have 
  \begin{equation}
    \label{eq:166}
     \End_{\gl_{m+n}} (M^\frakp(\udelta) \otimes V^{\otimes A})  \cong \End_{\gl_{m+n}} (M^\frakp(\udelta) \otimes V^{\otimes (r+t)})
  \end{equation}
as vector spaces; the dimension of the r.h.s.\ of \eqref{eq:166} can be computed counting standard tableaux, and is well-known to be $2^{r+t}(r+t)!$
  \end{proof}

Putting together the isomorphisms \eqref{eq:40} for all $A \in \Seq_{r,t}$ one gets the following:

\begin{corollary}
  \label{cor:5}
  With the hypotheses of Theorem~\ref{thm:2}, the cyclotomic regular monomials \eqref{eq:259} give a basis of $\uVB_{r,t}(\bomega; \beta_1,\beta_2;\beta_1^*,\beta_2^*)$ and we have an isomorphism of algebras
  \begin{equation}
    \label{eq:261}
    \Mor(\uVB_{r,t}(\bomega;\beta_1,\beta_2;\beta_1^*,\beta_2^*)) \cong \End_{\gl_{n+m}} \bigg(\bigoplus_{A \in \Seq_{r,t}} M^\frakp(\udelta) \otimes V^{\otimes A}\bigg).
  \end{equation}
\end{corollary}

We conclude this section giving an explicit formula to compute the parameters $\omega_k$ in term of $\beta_1,\beta_2$ (and hence in term of $m$, $n$ and $\delta$).

\begin{lemma}
  \label{lem:17}
  The elements $\omega_k$ in $\VB_{r,t}(\boldsymbol \omega; \beta_1, \beta_2; \beta_1^*,\beta_2^*)$ satisfy the following recursion formula
  \begin{equation}
    \label{eq:90}
    \omega_k = (\beta_1+ \beta_2) \omega_{k-1} - \beta_1 \beta_2 \omega_{k-2}    \end{equation}
with initial data $\omega_0=m+n$, $\omega_1= -\delta m+\frac{(m+n)^2}{2}$.
\end{lemma}

\begin{proof}
  By Lemma \ref{lem:1}, $\omega_0=m+n$. Let $A=(1,-1)$; we have
  \begin{equation}
    \label{eq:91}
    \begin{aligned}
      &e^{(A)}_{1} y_1 e^{(A)}_{1} (z \otimes v_i \otimes v_i^*)  =
      \sum_{j=1}^{m+n} e^{(A)}_{1} y_1 (z \otimes v_j \otimes
      v_j^*) \\
       & \quad = \sum_{j=1}^m e^{(A)}_{1} (-\delta z \otimes v_j \otimes v_j^*) + \frac{m+n}{2} \sum_{j=1}^{m+n} e^{(A)}_{1}(z \otimes v_j \otimes v_j^*)\\
      & \quad = \left(-\delta m + \frac{(m+n)^2}{2}\right) e^{(A)}_{1} (z \otimes v_j \otimes v_j^*),
    \end{aligned}
  \end{equation}
  hence $\omega_1=-\delta m + \frac{(m+n)^2}{2}$.

  The recursion relation follows from
  \begin{equation}
e^{(A)}_1 y_1^n e^{(A)}_{1}= (\beta_1 + \beta_2) e^{(A)}_{1} y_1^{n-1} e^{(A)}_{1} - \beta_1 \beta_2 e^{(A)}_{1} y_1^{n-2} e^{(A)}_{1}.\qedhere\label{eq:266}
\end{equation}
\end{proof}

Using the standard elementary theory of power series defined by recurrence relations, we can explictely compute $W_1(u)$:
\begin{equation}
  \label{eq:92}
  W_1(u)=\frac{\omega_0 + (\omega_1 - (\beta_1 + \beta_2)\omega_0)u^{-1}}{1-(\beta_1 + \beta_2) u^{-1} + \beta_1 \beta_2 u^{-2}}.
\end{equation}

\section{Diagram and partition calculus}
\label{sec:diagr-part-calc}

We explain now how one can determine which composition factors appear in $M^{\frakp}(\udelta) \otimes V^{\otimes A}$ using
a partition calculus.

Recall that a \emph{Young diagram} is a collection of boxes arranged in left-justified rows with a weakly decreasing number of boxes in each row. The \emph{content} of the box in the $r$-th row and $c$-th column (counting from the left to the right and from the top to the bottom, and starting with $0$) is $r-c$.

A \emph{rotated Young diagram} is a Young diagram rotated of 180 degrees. The content of a box in a rotated Young diagram is the same as the content of the original box in the original Young diagram (Figure \ref{fig:youngdiagrams}).

\begin{figure}
  \centering
  \begin{tikzpicture}[scale=0.7]
    \draw (0,0) -- (4,0);
    \draw (0,-1) -- (4,-1);
    \draw (0,-2) -- (3,-2);
    \draw (0,-3) -- (1,-3);
    \draw (0,0) -- (0,-3);
    \draw (1,0) -- (1,-3);
    \draw (2,0) -- (2,-2);
    \draw (3,0) -- (3,-2);
    \draw (4,0) -- (4,-1);
    \begin{scope}[xshift=0.5cm,yshift=-0.5cm]
      \node at (0,0) {$0$}; \node at (1,0) {$-1$}; \node at (2,0)
      {$-2$}; \node at (3,0) {$-3$}; \node at (0,-1) {$1$}; \node at
      (1,-1) {$0$}; \node at (2,-1) {$-1$}; \node at (0,-2) {$2$};
    \end{scope}
  \end{tikzpicture}
  \hspace{1cm}
  \begin{tikzpicture}[scale=0.7,rotate=180]
    \draw (0,0) -- (4,0);
    \draw (0,-1) -- (4,-1);
    \draw (0,-2) -- (3,-2);
    \draw (0,-3) -- (1,-3);
    \draw (0,0) -- (0,-3);
    \draw (1,0) -- (1,-3);
    \draw (2,0) -- (2,-2);
    \draw (3,0) -- (3,-2);
    \draw (4,0) -- (4,-1);
    \begin{scope}[xshift=0.5cm,yshift=-0.5cm]
      \node at (0,0) {$0$}; \node at (1,0) {$-1$}; \node at (2,0)
      {$-2$}; \node at (3,0) {$-3$}; \node at (0,-1) {$1$}; \node at
      (1,-1) {$0$}; \node at (2,-1) {$-1$}; \node at (0,-2) {$2$};
    \end{scope}
  \end{tikzpicture}
  \caption{A Young diagram and a rotated Young diagram, with the contents written in the boxes.}
  \label{fig:youngdiagrams}
\end{figure}
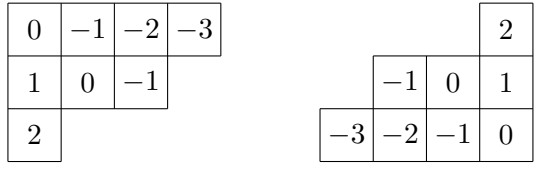

Let us now fix positive integers $m$ and $n$ and an integer $\delta$. In a plane let us consider the infinite vertical strip of width equal to $m+n$ boxes. Fix a horizontal line $o$. Fix also a vertical line $v$ in such a way that there is space for exactly $n$ boxes on the left of $v$ and for $m$ boxes on the right of $v$. The lines $o$ and $v$ divide our strip into four regions. Let us number the columns of our vertical strip by the integers $1,\ldots,m+n$ from the right to the left. We define a {\em 4-Young diagram} to be a collection of boxes in this strip, such that in the two regions under the horizontal line $o$ we have two Young diagrams and in the two regions above $o$ we have two rotated Young diagrams, and such that in no column there are boxes both above and under $o$ (see Figure \ref{fig:4youngdiagram}).

By definition a 4-Young diagram is made of four Young diagrams, and every box belongs to exactly one of these. We define the content of a box to be the content in the corresponding Young diagram, translated by the following values:
\begin{itemize}
\item the lower left Young diagram by $\frac{m+n}{2}$, 
\item the upper left Young diagram by $\frac{n-m}{2}$, 
\item the lower right Young diagram by $\frac{m-n}{2}+\delta$, 
\item the upper right Young diagram by $\frac{m+n}{2}-\delta$.
\end{itemize}
See Figure \ref{fig:4youngdiagram}.

\begin{figure}
\centering
\begin{tikzpicture}[scale=0.7]

  \draw[dashed] (0,4) -- (0,-4);
  \draw[dashed] (6,4) -- (6,-4);
  \draw[dashed] (12,4) -- (12,-4);
  \draw[dashed] (0,0) -- (12,0);

  \draw (0,0) -- (2,0);
  \draw (0,-1) -- (2,-1);
  \draw (0,-2) -- (1,-2);
  \draw (0,-3) -- (1,-3);
  \draw (0,0) -- (0,-3);
  \draw (1,0) -- (1,-3);
  \draw (2,0) -- (2,-1);
  \node at (0.5,-0.5) {$0$};
  \node at (1.5,-0.5) {$-1$};
  \node at (0.5,-1.5) {$1$};
  \node at (0.5,-2.5) {$2$};
  \node[anchor=north west] at (2,-2) {$\frac{m+n}{2}$};
  \draw[thin] (2,-2) -- (1.4,-1.4);

  \begin{scope}[xshift=6cm, rotate=180]
  \draw (0,0) -- (2,0);
  \draw (0,-1) -- (2,-1);
  \draw (0,-2) -- (1,-2);
  \draw (0,-3) -- (1,-3);
  \draw (0,0) -- (0,-3);
  \draw (1,0) -- (1,-3);
  \draw (2,0) -- (2,-1);
  \node at (0.5,-0.5) {$0$};
  \node at (1.5,-0.5) {$-1$};
  \node at (0.5,-1.5) {$1$};
  \node at (0.5,-2.5) {$2$};
  \node[anchor=south east] at (2,-2) {$\frac{n-m}{2}$};
  \draw[thin] (2,-2) -- (1.4,-1.4);
  \end{scope}

  \begin{scope}[xshift=6cm]
  \draw (0,0) -- (2,0);
  \draw (0,-1) -- (2,-1);
  \draw (0,-2) -- (1,-2);
  \draw (0,-3) -- (1,-3);
  \draw (0,0) -- (0,-3);
  \draw (1,0) -- (1,-3);
  \draw (2,0) -- (2,-1);
  \node at (0.5,-0.5) {$0$};
  \node at (1.5,-0.5) {$-1$};
  \node at (0.5,-1.5) {$1$};
  \node at (0.5,-2.5) {$2$};
  \node[anchor=north west] at (2,-2) {$\frac{m-n}{2}+\delta$};
  \draw[thin] (2,-2) -- (1.4,-1.4);
  \end{scope}

  \begin{scope}[xshift=12cm,rotate=180]
  \draw (0,0) -- (2,0);
  \draw (0,-1) -- (2,-1);
  \draw (0,-2) -- (1,-2);
  \draw (0,-3) -- (1,-3);
  \draw (0,0) -- (0,-3);
  \draw (1,0) -- (1,-3);
  \draw (2,0) -- (2,-1);
  \node at (0.5,-0.5) {$0$};
  \node at (1.5,-0.5) {$-1$};
  \node at (0.5,-1.5) {$1$};
  \node at (0.5,-2.5) {$2$};
  \node[anchor=south east] at (2,-2) {$\frac{m+n}{2}-\delta$};
  \draw[thin] (2,-2) -- (1.4,-1.4);
  \end{scope}

\end{tikzpicture}
\caption{A 4-Young diagram with the contents in the boxes. The corresponding weight is $3 \epsilon_1 + \epsilon_2 - \epsilon_{m-1} - 3 \epsilon_m + 3\epsilon_{m+1}+\epsilon_{m+2} - \epsilon_{m+n-1} - 3 \epsilon_{m+n}$. }
\label{fig:4youngdiagram}
\end{figure}
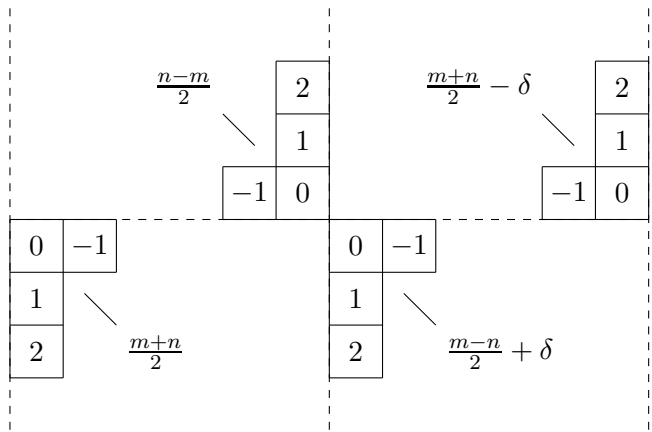

Given a 4-Young diagram $Y$, let $b_i(Y)$ be equal to the number of boxes in the column $i$ of $Y$, multiplied by $-1$ if the boxes are under the horizontal line $o$.

To $\lambda \in \Lambda(m,n)$ we associate the 4-Young diagram $Y(\lambda)$ determined by $b_i(Y(\lambda))=\langle\lambda,\epsilon_i\rangle$. More generally, in the same way, to any weight for $\gl_{m+n}$ we could associate a diagram consisting of boxes in our infinite vertical strip: one can check that this diagram is a 4-Young diagram if and only if the weight is in $\Lambda(m,n)$.

Given a 4-Young diagram $Y$, we may obtain another 4-Young diagram
$Y'$ by adding a box to it (we also say that $Y$ is obtained by
removing a box to $Y'$).  Notice that we use the expressions \emph{adding} and \emph{removing
boxes} only if the result is again a 4-Young diagram.  For a
$(r,t)$-sequence $A$ let us define $\calY_A$ to be the set of
sequences
\begin{equation}
  Y_0=Y(\udelta),Y_1,\ldots,Y_{r+t}\label{eq:236}
\end{equation}
of 4-Young diagrams such
that $Y_{i+1}$ is obtained from $Y_i$ by
\begin{itemize}
\item adding a box above $o$ or removing a box below $o$ if $a_i = 1$,
\item removing a box above $o$ or adding a box below $o$ if $a_i=-1$.
\end{itemize}

From the construction and the properties of the functors of tensoring with $V$ and $V^*$ on $\catO(m,n)$ we have the following result:

\begin{lemma}
  \label{lem:40}
  There is a bijection between $\calY_A$ and the composition factors
  of $M^\frakp(\udelta)\otimes V^{\otimes A}$ (counted with
  multiplicity). The element $Y_\bullet \in \calY_A$ corresponds to a
  composition factor isomorphic to $M(\lambda)$.
\end{lemma}

Notice that as a consequence we have the following non-trivial combinatorial statement:

\begin{corollary}
  \label{cor:3}
  The cardinality of $\calY_A$ is $2^{r+t} (r+t)!$.
\end{corollary}

We can now compute the generalized eigenvalues of the $y_i$'s:

\begin{prop}
  \label{prop:4}
  Let $Y_\bullet \in \calY_A$. For $j=1,\ldots,m+n$ let $\eta_j=1$ if $Y_j$ has been obtained from $Y_{j-1}$ by adding a box of content $i_j$, otherwise let $\eta_j=-1$ if $Y_j$ has been obtained from $Y_{j-1}$ by removing a box of content $i_j$. Then the corresponding composition factor $M(\udelta+(-1)^{a_1} \epsilon_{\kappa_1} + \cdots +(-1)^{a_{r+t}} \epsilon_{\kappa_{r+t}})$ of $M^\frakp(\udelta) \otimes V^{\otimes A}$ is contained in the generalized eigenspace for the $y_i$'s with generalized eigenvalues $(\eta_1 i_1,\ldots, \eta_{r+t} i_{r+t})$.
\end{prop}
\begin{proof}
  Remember that we denote by $C$ the Casimir element of $\gl_{m+n}$. We have
  \begin{equation}
    \label{eq:104}
    \Delta^\ell(C) = \Delta^{\ell-1}(C) \otimes 1 + \Delta^{\ell-1}(1) \otimes C + 2 \sum_{p=0}^{\ell-1} \Omega_{p \ell}.
  \end{equation}
  Suppose first that $a_\ell=1$. To simplify the notation we set
  \begin{equation}
    \label{eq:232}
    \psi_h = (-1)^{a_1} \epsilon_1 + \cdots + (-1)^{a_h} \epsilon_{\kappa_h}
  \end{equation}
  for all $h=1,\ldots,r+t$.

 Then the action of $2\sum_{p=0}^{\ell-1} \Omega_{p \ell}$ on $M(\udelta+(-1)^{a_1} \epsilon_{\kappa_1} + \cdots +(-1)^{a_{r+t}} \epsilon_{\kappa_{r+t}}) = M(\udelta + \psi_{r+t})$ is given by
  \begin{equation}
    \label{eq:105}
      \langle \udelta+ \psi_\ell , \udelta+ \psi_\ell + 2\rho \rangle - \langle \udelta+ \psi_{\ell-1} , \udelta+ \psi_{\ell -1} + 2\rho\rangle - \langle \epsilon_1 , \epsilon_1 + 2\rho \rangle.
  \end{equation}
  Hence $y_\ell$ acts as
  \begin{equation}
    \label{eq:106}
    \langle \epsilon_{\kappa_\ell} ,  \udelta\rangle +  \langle \epsilon_{\kappa_\ell}, \psi_{\ell-1} \rangle + \langle \epsilon_{\kappa_\ell} - \epsilon_1,  \rho \rangle + \frac{m+n}{2}.
  \end{equation}

  Now suppose instead that $a_\ell=-1$. Then  the action of $2\sum_{p=0}^{\ell-1} \Omega_{p \ell}$ on $M(\udelta + \psi_{r+t})$ is given by
  \begin{equation}
    \label{eq:105}
      \langle \udelta+ \psi_\ell , \udelta+ \psi_\ell + 2\rho \rangle - \langle \udelta+ \psi_{\ell-1} , \udelta+ \psi_{\ell -1} + 2\rho\rangle + \langle \epsilon_{m+n} ,- \epsilon_{m+n} + 2\rho \rangle.
  \end{equation}
  Hence $y_\ell$ acts as
  \begin{equation}
    \label{eq:246}
    \begin{aligned}
      & -\langle \epsilon_{\kappa_\ell} , \udelta\rangle - \langle
      \epsilon_{\kappa_\ell}, \psi_{\ell-1} \rangle + \langle
      \epsilon_{m+n} - \epsilon_{k_\ell}, \rho \rangle + \frac{m+n}{2}\\
      & = -\langle \epsilon_{\kappa_\ell} , \udelta\rangle - \langle
      \epsilon_{\kappa_\ell}, \psi_{\ell-1} \rangle - \langle
      \epsilon_{\kappa_\ell} - \epsilon_{1}, \rho \rangle - \frac{m+n}{2} + 1.
    \end{aligned}
  \end{equation}

  Let us now examine \eqref{eq:106} and \eqref{eq:246}: in both of them, the second term is the index and the third term is the is (up to a shift) the column index of the box being added/removed at the $\ell$-th step. The claim then follows by the definition of the shifts of the contents of the boxes in the 4-Young diagram.
\end{proof}

\section{Proofs}
\label{sec:proofs}

We collect in this section the steps of the proof of Theorem \ref{thm:1}.

First, note that
\begin{equation}
  \label{eq:21}
  \dot s_i \Omega_{jk} = \Omega_{\mathsf s_i(j) \mathsf s_i (k)} \dot s_i
\end{equation}
for every $0 \leq j < k \leq r+t$,
where we define $\Omega_{kj} = \Omega_{jk}$ for $k>j$.

\begin{lemma}
  \label{lem:2}
  For all $\gl_N$-modules $M$ and all $A \in \Seq_{r,t}$, on $M \otimes V^{\otimes A}$ we have $\dot s_i y_j=y_j \dot s_i$ for $j \neq i, i+1$.
\end{lemma}

\begin{proof}
  The statement is obvious for $j < i$. Suppose $j > i+1$: we have
  \begin{equation*}
      \dot s_i y_j = \dot s_i \bigg[ \sum_{0 \leq k < j} \Omega_{kj} +
      \frac{N}{2} \bigg] =  \bigg[ \sum_{0 \leq k < j} \Omega_{\mathsf s_i(k)j} +
      \frac{N}{2} \bigg] \dot s_i = y_j \dot s_i. \qedhere 
  \end{equation*}
\end{proof}

\begin{lemma}
  \label{lem:4}
  On $V \otimes V^* \otimes V$ and on $V \otimes V^* \otimes V^*$ the elements $(\Omega_{13} + \Omega_{23})(\tau_V \otimes \id)$ and $(\tau_V \otimes \id)(\Omega_{13} + \Omega_{23})$ act as zero. Analogously, on $V^* \otimes V \otimes V$ and on $V^* \otimes V \otimes V^*$ the elements $(\Omega_{13} + \Omega_{23})(\tau_{V^*} \otimes \id)$ and $(\tau_{V^*} \otimes \id)(\Omega_{13} + \Omega_{23})$ act as zero.
\end{lemma}

\begin{proof}
  We prove only that the two actions on $V \otimes V^* \otimes V$ are zero. The other assertions follow analogously.
We compute:
\begin{align*}
  & (\Omega_{ij}+ \Omega_{(i+1)j}) (\tau_V \otimes \id) (v_a \otimes
  v^*_b \otimes v_c) \\ \displaybreak[2] &\qquad = (\Omega_{ij}+
  \Omega_{(i+1)j}) \big(\delta_{ab} \sum_{d \in I} v_d \otimes v^*_d
  \otimes v_c\big)\\ \displaybreak[2]
  & \qquad =  \delta_{ab} \sum_{d,e \in I} \big( H_e v_d \otimes v_d^* \otimes H_e v_c + v_d \otimes H_e v_d^* \otimes H_e v_c \big)  + \\
  & \qquad \qquad + \delta_{ab} \sum_{d \in I,\alpha \in \Pi} \big(
  X_\alpha v_d \otimes v_d^* \otimes X_{-\alpha} v_c + v_d \otimes
  X_\alpha v_d^* \otimes X_{-\alpha} v_c \big) \\ \displaybreak[2]
  & \qquad =  \delta_{ab} ( v_c \otimes v_c^* \otimes v_c - v_c \otimes v_c^* \otimes v_c )   +\\
  &\qquad \qquad + \delta_{ab} \sum_{d \in I} \big( (1-\delta_{cd})(
  v_c \otimes v_d^* \otimes v_d) - \delta_{cd}\sum_{e\in I, e \neq c}(
  v_d \otimes v_e^* \otimes v_e )\big)\\ \displaybreak[2]
  &\qquad =
  \delta_{ab} \sum_{d \neq c} (v_c \otimes v_d^* \otimes v_d) -
  \delta_{ab} \sum_{e \neq c} (v_c \otimes v_e^* \otimes v_e) = 0.
\end{align*}

Next compute:
\begin{align*}
&    (\tau_V \otimes \id)(\Omega_{ij}+ \Omega_{(i+1)j}) (v_a \otimes v^*_b \otimes v_c)\\ \displaybreak[2]
& \qquad =  (\tau_V \otimes \id) \biggl[ \sum_{d \in I} \big( H_d v_a \otimes v_b^* \otimes H_d v_c + v_a \otimes H_d v_b^* \otimes H_d v_c \big)  + \\ \displaybreak[2]
& \qquad \qquad +  \sum_{\alpha \in \Pi} \big( X_\alpha v_a \otimes v_b^* \otimes X_{-\alpha} v_c + v_a \otimes X_\alpha v_b^* \otimes X_{-\alpha} v_c \big) \biggr] \\ \displaybreak[2]
& \qquad =  (\tau_V \otimes \id) \big[ \big( \delta_{ac} v_a \otimes v_b^* \otimes v_a - \delta_{bc} v_a \otimes v_b^* \otimes v_b \big)  +\\ 
&\qquad \qquad +  \big( (1-\delta_{ac} )(v_c \otimes v_b^* \otimes v_a ) - \delta_{bc}\sum_{d \in I, d \neq b} (v_a \otimes v_d^* \otimes v_d \big) \big] \\ \displaybreak[2]
& \qquad =  \sum_{e\in I} \delta_{ac}\delta_{ab} v_e \otimes v_e^* \otimes v_a -  \sum_{e \in I} \delta_{bc}\delta_{ab} v_e \otimes v_e^* \otimes v_b +\\
&\qquad \qquad + \sum_{e\in I} \delta_{cb}(1-\delta_{ac})(v_e \otimes v_e^* \otimes v_a) - \delta_{bc} \sum_{d \in I, d \neq b} \delta_{ad} \sum_{e \in I} (v_e \otimes v_e^* \otimes v_d)\\ \displaybreak[2]
& \qquad = \delta_{bc}(1- \delta_{ac})\sum_{e \in I} (v_e \otimes v_e^* \otimes v_a)  -
\delta_{bc}(1- \delta_{ab}) \sum_{e \in I} (v_e \otimes v_e^* \otimes v_a) = 0.\qedhere
  \end{align*}
\end{proof}

\begin{lemma}
  \label{lem:3}
  For all modules $M$ and $(r,t)$-sequences $A$, on $M \otimes
  V^{\otimes A}$ we have $ e_i y_j = y_j e_i$ for $j \neq i,i+1$.
\end{lemma}

\begin{proof}
  The case $j<i$ is obvious, and we are left with the case $j> i+1$. It suffices to prove that
  \begin{equation*}
    (\Omega_{ij}+ \Omega_{(i+1)j})e_i = e_i (\Omega_{ij}+ \Omega_{(i+1)j}),
  \end{equation*}
  since $e_i$ commutes with all other summands of $y_j$. This follows immediately from Lemma \ref{lem:4}.
\end{proof}

\begin{lemma}
  \label{lem:5}
  We have
    $[\Omega_{12}, \Omega_{34}]=0$ and $[\Omega_{12}+\Omega_{23},\Omega_{13}]=0$.
\end{lemma}
\begin{proof}
  The first equation is obvious. For the second, we explicitly compute the expression $\Omega_{12}\Omega_{13}+\Omega_{23}\Omega_{13}-\Omega_{13}\Omega_{12}-\Omega_{13}\Omega_{23}$ and we get:
  \begin{align*}
    &\sum_{\substack{a,b,c \in I\\  b \neq c}} \big( H_a X_{bc} \otimes H_a \otimes X_{{cb}} + X_{{bc}}\otimes H_a \otimes H_a X_{{cb}} -\\ 
& \qquad \qquad - H_a X_{bc} \otimes X_{{cb}} \otimes H_a - H_a \otimes X_{bc} \otimes H_a X_{{cb}} \big) + \\ 
    & \qquad + \sum_{a,b,c \in I, b \neq c} \big( X_{bc} H_a \otimes X_{{cb}} \otimes H_a + H_a\otimes X_{bc} \otimes X_{{cb}} H_a -\\ 
& \qquad \qquad - X_{bc} H_a \otimes H_a \otimes X_{{cb}} -X_{bc} \otimes H_a \otimes X_{{cb}}H_a \big) + \\ 
& \qquad + \sum_{\substack{a,b,c,d \in I\\ a \neq b, c \neq d}}\big( X_{ab} X_{cd} \otimes X_{{ba}} \otimes X_{{dc}} + X_{cd} \otimes X_{ab} \otimes X_{{ba}} X_{{dc}} -\\ 
& \qquad \qquad - X_{ab} X_{cd} \otimes X_{{dc}} \otimes X_{{ba}} - X_{ab} \otimes X_{cd} \otimes X_{{ba}} X_{{dc}} \big) \\ \displaybreak[2]
& = \sum_{\substack{a,b,c \in I\\  b \neq c}} \big( [H_a,X_{bc}] \otimes H_a \otimes X_{{cb}} + X_{bc} \otimes H_a \otimes [H_a,X_{{cb}}]- \\ 
& \qquad \qquad - [H_a,X_{bc}] \otimes X_{{cb}} \otimes H_a - H_a \otimes X_{bc} \otimes [H_a,X_{{cb}}]\big) + \\ 
& \qquad + \sum_{\substack{a,b,c,d \in I\\ a \neq b, c \neq d}} \big([X_{ab},X_{cd}] \otimes X_{ba} \otimes X_{dc} - X_{ab} \otimes X_{cd} \otimes [X_{ba},X_{dc}]\big)\\ \displaybreak[2]
& = \sum_{\substack{b,c \in I\\ b \neq c}} \big( (X_{bc} - X_{bc}) \otimes H_a \otimes X_{{cb}} + X_{bc} \otimes H_a \otimes (X_{{cb}}-X_{cb})- \\
& \qquad \qquad - (X_{bc}- X_{bc}) \otimes X_{{cb}} \otimes H_a - H_a \otimes X_{bc} \otimes (X_{cb}-X_{cb})\big) + \\ 
& \qquad+ \sum_{\substack{a,b,c,d \in I\\ a \neq b, c \neq d}} \big( (\delta_{bc}(1-\delta_{ad})X_{ad}-\delta_{ad}(1- \delta_{bc})X_{cb}) \otimes X_{ba} \otimes X_{dc} -\\ 
& \qquad \qquad -  X_{ab} \otimes X_{cd} \otimes (\delta_{ad}(1-\delta_{bc}) X_{bc} - \delta_{bc}(1-\delta_{ad} )X_{da})\big)\\ \displaybreak[2]
& = \sum_{\substack{a,b,d \in I\\ a \neq b, b \neq d, a \neq d}} ( X_{ad} \otimes X_{ba} \otimes X_{db} )-  \sum_{\substack{a,b,c \in I\\ a \neq b, a \neq c, b \neq c}} (X_{cb} \otimes X_{ba} \otimes X_{ac}) -\\ 
& \qquad -   \sum_{\substack{a,b,c \in I\\ a \neq b, a \neq c, b \neq c}} (X_{ab} \otimes X_{ca} \otimes X_{bc}) +   \sum_{\substack{a,b,d \in I \\ a \neq b, b \neq d, a \neq d}}  (X_{ab} \otimes X_{bd} \otimes X_{da}) = 0,
    \end{align*}
    that proves our claim.
\end{proof}

\begin{lemma}
  \label{lem:6}
  For $1 \leq i,j \leq r+t$ we have $y_i y_j=y_j y_i$.
\end{lemma}

\begin{proof}
  Let us assume $i>j$, and compute
  \begin{align*}
      [y_i,y_j]&= \left[ \sum_{0 \leq h < i} \Omega_{hi}, \sum_{0 \leq k < j} \Omega_{kj}\right] =  \sum_{0 \leq k < j} \left[ \sum_{0 \leq h < i} \Omega_{hi},  \Omega_{kj}\right] \\ \displaybreak[2]
&=  \sum_{0 \leq k < j} \left[ \Omega_{ki} + \Omega_{ji},  \Omega_{kj}\right] = 0
    \end{align*}
  using Lemma \ref{lem:5}.
\end{proof}

\begin{lemma}
  \label{lem:7}
  Relations $\dot s_i \dot e_{i+1} \dot e_i = \dot s_{i+1} \dot e_i$,
  $\dot e_i \dot e_{i+1} \dot s_i = \dot e_i \dot s_{i+1}$,
  $\dot e_{i+1} \dot e_i \dot s_{i+1} = \dot e_{i+1} \dot s_i$,
 $\dot s_{i+1} \dot e_i \dot e_{i+1} = \dot s_i \dot e_{i+1}$,
 $\dot e_{i+1} \dot e_i \dot e_{i+1} = \dot e_{i+1} $ and
          $\dot e_i \dot e_{i+1} \dot e_i = \dot e_i$ hold for $1 \leq i \leq r+t-2$.
\end{lemma}

\begin{proof}
  These relations are very easy to check by hand. Alternatively, they are implied by the standard (trivial) ribbon Hopf algebra structure of $U(\gl_N)$.
\end{proof}

\begin{lemma}
  \label{lem:9}
  We have $\sigma_{V,V}= \Omega$ on $V \otimes V$ and $\sigma_{V^*,V^*} = \Omega$ on $V^* \otimes V^*$.
\end{lemma}

\begin{proof}
  We compute
  \begin{align*}
      \Omega (v_i \otimes v_j) & = \sum_{a \in I} H_a v_i \otimes H_a v_j + \sum_{\alpha \in \Pi} X_\alpha v_i \otimes X_{-\alpha} v_j \\ \displaybreak[2]
      &=\delta_{ij} v_i \otimes v_i + (1-\delta_{ij}) v_j \otimes v_i = v_j \otimes v_i.
    \end{align*}
  Similarly we obtain the second equality.
\end{proof}

\begin{lemma}
  \label{lem:8}
  We have $s_i y_i - y_{i+1} s_i = - 1$ and $ s_i y_{i+1} - y_i s_i = 1$.
\end{lemma}

\begin{proof}
  This is a standard fact (see \cite[Lemma 2.1]{MR1652134}), but we repeat the proof for completeness. We compute, using \eqref{eq:21} and Lemma \ref{lem:9}:
  \begin{align*}
      s_i \sum_{0 \leq k < i} \Omega_{ki} - \sum_{0 \leq k < i+1}
      \Omega_{k(i+1)} s_i& = \sum_{0 \leq k < i} \Omega_{k(i+1)}s_i -
      \sum_{0 \leq k < i+1} \Omega_{k(i+1)} s_i \\ \displaybreak[2]
      &= - \Omega_{i(i+1)}s_i = -s_i^2 = -1
    \end{align*}
and
  \begin{align*}
      \sum_{0 \leq k < i} \Omega_{ki} s_i - s_i \sum_{0 \leq k < i+1}
      \Omega_{k(i+1)} & = \sum_{0 \leq k < i} \Omega_{ki}s_i -
      \sum_{0 \leq k < i} \Omega_{ki} s_i - s_i \Omega_{i(i+1)}\\ \displaybreak[2]
      &= - s_i^2 = -1.\qedhere
    \end{align*}
\end{proof}

\begin{lemma}
  \label{lem:10}
  We have $\tau_{V}= - \Omega$ on $V \otimes V^*$ and $\tau_{V^*} = -\Omega$ on $V^* \otimes V$.
\end{lemma}

\begin{proof}
  We compute
  \begin{align*}
      \Omega (v_i \otimes v^*_j) & = \sum_{a \in I} H_a v_i \otimes H_a v^*_j + \sum_{\alpha \in \Pi} X_\alpha v_i \otimes X_{-\alpha} v^*_j \\ \displaybreak[2]
      &= - \delta_{ij} v_i \otimes v^*_i - \delta_{ij}\sum_{k \neq i} v_k \otimes v^*_k = - \delta_{ij} \sum_{k \in I} v_k \otimes v^*_k.
    \end{align*}
  Similarly we obtain the second equality.
\end{proof}

\begin{lemma}
  \label{lem:11}
  We have $\hat s_i y_i - y_{i+1} \hat s_i = \hat e_i$ and $ \hat s_i y_{i+1} - y_i \hat s_i = -\hat e_i$.
\end{lemma}

 \begin{proof}
 We compute, using \eqref{eq:21} and Lemma \ref{lem:10}:
   \begin{align*}
       \hat s_i \sum_{0 \leq k < i} \Omega_{ki} - \sum_{0 \leq k < i+1}
       \Omega_{k(i+1)} \hat s_i& = \sum_{0 \leq k < i} \Omega_{k(i+1)} \hat s_i -
       \sum_{0 \leq k < i+1} \Omega_{k(i+1)} \hat s_i \\ \displaybreak[2]
       &= - \Omega_{i(i+1)} \hat s_i = e_i \hat s_i = \hat e_i
     \end{align*}
and
  \begin{align*}
      \sum_{0 \leq k < i} \Omega_{ki} \hat s_i - \hat s_i \sum_{0 \leq k < i+1}
      \Omega_{k(i+1)} & = \sum_{0 \leq k < i} \Omega_{ki} \hat s_i -
      \sum_{0 \leq k < i} \Omega_{ki} \hat s_i - \hat s_i \Omega_{i(i+1)}\\ \displaybreak[2]
      &= \hat s_i e_i = \hat e_i.\qedhere
    \end{align*}
  \end{proof}

\begin{lemma}
  \label{lem:12}
  Let $M$ be a highest weight module with one-dimensional highest weight space. Then for all $1 \leq i < r+t$ we have $e_i (y_i + y_{i+1}) = 0$ and $(y_i + y_{i+1})e_i = 0$.
\end{lemma}

\begin{proof}
  We start expanding the first relation:
  \begin{align*}
      e_i (y_i + y_{i+1}) & = e_i \left( \sum_{0 \leq k < i} ( \Omega_{ki} + \Omega_{k(i+1)} ) + \Omega_{i(i+1)} + N \right)\\ \displaybreak[2]
      &= e_i \left( \sum_{0 \leq k < i} ( \Omega_{ki} + \Omega_{k(i+1)} )\right)  - e_i^2  + N e_i.
    \end{align*}
  We know that $e_i^2 = N e_i$. Moreover, by Lemma \ref{lem:4} we know
  that $e_i (\Omega_{ki}+ \Omega_{k(i+1)})$ acts as $0$ if $k
  >0$. Hence we are left to show that $e_i (\Omega_{0i} +
  \Omega_{0(i+1)})$ acts as $0$. Let $m$ be a non-zero vector in the
  highest weight space of $M$ and suppose that at the place $i$ we
  have $V$ and at the place $i+1$ we have $V^*$ (the other case being
  analogous). We write only the factors $0, i$ and $i+1$ of the
  tensor product, and compute $ e_i (\Omega_{0i} + \Omega_{0(i+1)}) (m \otimes v_h \otimes v^*_k)$:
  \begin{align*}
     & e_i \sum_{a \in I} \big( H_a m \otimes H_a v_h \otimes v_k^* + H_a m
      \otimes v_h \otimes H_a v_k^* \big) + \\ 
&\qquad + e_i \sum_{\alpha \in \Pi^-} \big( X_\alpha m \otimes X_{-\alpha} v_h \otimes v^*_k + X_\alpha m \otimes v_h \otimes X_{-\alpha} v^*_k\big)\\ \displaybreak[2]
&= \delta_{hk} \sum_{b \in I} \big( H_h m \otimes v_b \otimes v_b^* - H_h m
      \otimes v_b \otimes v_b^* \big) + \\ 
&\qquad + e_i  (1-\delta_{hk}) \big( X_{hk}  m \otimes  v_k \otimes v^*_k  - X_{hk} m \otimes v_h \otimes v_h\big)=0,
    \end{align*}
where in the fourth line we have written only the terms from the second line that survive after applying $e_i$.

Analogously, for the second relation, we consider
  \begin{align*}
      (y_i + y_{i+1})e_i & =  \left( \sum_{0 \leq k < i} ( \Omega_{ki} + \Omega_{k(i+1)} ) + \Omega_{i(i+1)} + N \right)e_i\\ \displaybreak[2]
      &= \left( \sum_{0 \leq k < i} ( \Omega_{ki} + \Omega_{k(i+1)} )\right)e_i   - e_i^2  + N e_i
  \end{align*}
and as before we just need to compute $(\Omega_{0i} + \Omega_{0(i+1)}) e_i (m \otimes v_h \otimes v^*_k)$: this can be non-zero only if $h=k$ and in this case we get
  \begin{align*}
     & (\Omega_{0i}+ \Omega_{0(i+1)} ) \left( \sum_{l \in I} m \otimes v_l \otimes v_l^*\right)\\ \displaybreak[2]
&= \sum_{a,l \in I} \big( H_a m \otimes H_a v_l \otimes v_l^* + H_a m
      \otimes v_l \otimes H_a v_l^* \big) + \\ 
&\qquad + \sum_{\alpha \in \Pi^-} \big( X_\alpha m \otimes X_{-\alpha} v_l \otimes v^*_l + X_\alpha m \otimes v_l \otimes X_{-\alpha} v^*_l\big)\\ \displaybreak[2]
&=  \sum_{\substack{b,l \in I\\ b\neq l}} \big( X_{bl} m \otimes  v_b \otimes v^*_l - X_{bl} m \otimes v_l \otimes  v^*_b\big) = 0. \qedhere
  \end{align*}
\end{proof}

\end{document}